\documentclass[10pt]{article}
\usepackage{amsfonts}
\usepackage{amsmath,hhline,epsfig,latexsym}
\usepackage{amssymb}
\usepackage{graphicx}
\usepackage[english]{babel}
\usepackage{hyperref}
\usepackage{xcolor}
\makeatletter
\renewcommand{\paragraph}{\@startsection{paragraph}{4}{0ex}%
   {-3.25ex plus -1ex minus -0.2ex}%
   {1.5ex plus 0.2ex}%
   {\normalfont\normalsize\bfseries}}
\makeatother

\stepcounter{secnumdepth}
\stepcounter{tocdepth}
\usepackage{amsmath}
\providecommand{\U}[1]{\protect\rule{.1in}{.1in}}
\newtheorem{thm}{Theorem}
\newtheorem{prop}{Proposition}

\newtheorem{rmk}{Remark}

\newtheorem{lem}{Lemma}

\newenvironment{proof}[1][Proof]
{\noindent\textbf{#1:} }{\hfill\rule{0.5em}{0.5em}}

\newcommand{\Fin}{\hfill\rule{0.5em}{0.5em}}
\newtheorem{theorem}{Theorem}[section]
\newtheorem{remark}[theorem]{Remark}

\begin{document}


\title{\textbf{Remarks on the control of two-phase Stefan free-boundary problems}}



\author{
\textsc{Raul K.C. Ara\'ujo}\thanks{Department of Mathematics, Federal University of Pernambuco, UFPE, CEP 50740-545, Recife,
PE, Brazil. E-mail: {\tt  raul@dmat.ufpe.br}. Partially supported by CNPq (Brazil).}
\quad
\and
	\textsc{Enrique Fern\' andez-Cara}\thanks{University of Sevilla, Dpto. E.D.A.N, Aptdo 1160, 41080 Sevilla, Spain. E-mail: {\tt cara@us.es}. Partially supported by grant MTM$2016$-$76990$-P, Ministry of Economy and Competitiveness (Spain).}
	\and
	\textsc{Juan L\'imaco}\thanks{Inst. Matemática, Universidade Federal Fluminense, Niterói, RJ, Brazil. E-mail: {\tt jlimaco@vm.uff.br}.}
	\and
	\textsc{Diego A. Souza}\thanks{University of Sevilla, Dpto. E.D.A.N, Aptdo 1160, 41080 Sevilla, Spain. E-mail: {\tt desouza@us.es}. Partially supported by grant MTM$2016$-$76990$-P, Ministry of Economy and Competitiveness (Spain).}
}
\date{ }
\maketitle


\begin{abstract}
        This paper concerns the null controllability of the two-phase 1D Stefan problem with distributed controls. 
        This is a free-boundary problem that models solidification or melting processes.
        In each phase, a parabolic equation, completed with initial and boundary conditions must be satisfied;
        the phases are separated by a phase-change interface where an additional free-boundary condition is imposed.
        We assume that two localized sources of heating/cooling controls act on the system (one in each phase).
        We prove a local null controllability result: the temperatures and the interface can be respectively steered to zero and to a prescribed location provided 
        the initial data and interface position are sufficiently close to the targets. The ingredients of the proofs are a compactness-uniqueness argument (to deduce 
        appropriate observability estimates adapted to constraints) and a fixed-point formulation and resolution of the controllability problem (to deduce the result for the nonlinear system).
	We also prove a negative result corresponding to the case where only one control acts on the system and the interface does not collapse to the boundary.
\end{abstract}

\


\tableofcontents


\section{Introduction}

	The two-phase Stefan problem is a mathematical model (a coupled system composed of two PDEs and one ODE) used to describe liquid-solid phase transition processes.
	These processes appear frequently in science and engineering like, for example, in the continuous casting of steel \cite{petrus}, cancer treatment by cryosurgeries \cite{rabin}, 
	crystal growth \cite{conrad}, lithium-ion batteries \cite{camacho}, among others. It is also important to highlight that, besides the interest in thermodynamics processes, 
	similar systems represent models for phenomena of other kinds: analysis and computation of the flux in free surfaces \cite{brebbia, hermans, stoker}, fluid-solid 
	interaction \cite{fluid, liu, vazquez}, gases flow through porous medium \cite{aronson, fasano, vazquez1}, growth of tumors and others  mathematical modelling in biology
	 \cite{friedman2, friedman3}. 
   
	Let us present the mathematical formulation of the two-phase Stefan problem: let $L>0$, $T > 0$ and $\ell_l, \ell_0, \ell_r \in (0,L)$ be given with $\ell_l<\ell_0 < \ell_r$.  
	We also consider two functions $u_0\in H_0^1(0,\ell_0)$ with $u_0 \geq 0$ and $v_0\in H_0^1(\ell_0,L)$ with $v_0 \leq 0$ and two open sets $\omega_l\subset\subset 
	(0,\ell_l)$ and $\omega_r\subset\subset (\ell_r,L)$. At each $t$, the material domain is separated in two parts: $x\in [0, \ell(t))$ (liquid phase) and $x\in (\ell(t), L]$ (solid phase).
	Here, $\ell = \ell(t)$ is the position of the interface between liquid and solid phases; it satisfies  $\ell(0) = \ell_0$ and $\ell(t) \in (\ell_l,\ell_r)$  for all $t$. 
	
	The aim of this paper is to study the controllability properties of the {\it two-phase Stefan problem}:      
        \begin{equation}\label{am}
	\left\{
		\begin{array}{lll}
			u_t - d_lu_{xx} = h_l1_{\mathcal{\omega}_l} 	&\mbox{in}& 	Q_l,		\\
			\noalign{\smallskip}\displaystyle	
      			v_t - d_rv_{xx} = h_r1_{\mathcal{\omega}_r} 	&\mbox{in}& 	Q_r,		\\
			\noalign{\smallskip}\displaystyle	
      			u(0,t)  =0  				&\mbox{on}& 	(0,T),	\\
			\noalign{\smallskip}\displaystyle	
      			 v(L,t) = 0 				&\mbox{on}& 	(0,T),	\\
			\noalign{\smallskip}\displaystyle	
      			u(\cdot,0) = u_0  				&\mbox{in}& 	(0,\ell_0),	\\ 
			\noalign{\smallskip}\displaystyle	
      			v(\cdot,0) = v_0 				&\mbox{in}& 	(\ell_0,L), 	\\
			\noalign{\smallskip}\displaystyle	
			u(\ell(t),t) =v(\ell(t),t) =0  				&\mbox{on}& 	(0,T),	\\
			\noalign{\smallskip}\displaystyle
			- \ell'(t) = d_lu_x(\ell(t),t) - d_rv_{x}(\ell(t),t)	&\mbox{in}&	(0,T).
		\end{array}
	\right.
        \end{equation}
	Here and in the sequel, $d_l$ and $d_r$ must be viewed as diffusion coefficients and we use the notation
        \[
	\left\{
		\begin{array}{l}
        			Q := (0,L)\times (0,T),\\
			\noalign{\smallskip}\displaystyle	
        			Q_l := \{(x,t)\in Q:\,\, t\in (0,T),\,\, x\in (0,\ell(t))\},\\
			\noalign{\smallskip}\displaystyle	
        			Q_r := \{(x,t)\in Q:\,\, t\in (0,T),\,\, x\in (\ell(t),L)\},\\
			\noalign{\smallskip}\displaystyle	
			\mathcal{O}_l =  \omega_l\times (0,T)\,\,\mbox{and}\,\,\mathcal{O}_r = \omega_r\times (0,T).
		\end{array}
	\right.
        \]

 
     The main result in this paper is the following:     		
\begin{thm}\label{twophase}
	Let $\ell_T\in (\ell_l,\ell_r)$. Then there exists $\delta > 0$ such that, for any $u_0\in H_0^1(0,\ell_0)$ with $u_0 \geq 0$, any $v_0\in H_0^1(\ell_0,L)$ with $v_0 \leq 0$ and any 
	$\ell_0\in (\ell_l,\ell_r)$ satisfying 
        \[
	\|u_0\|_{H_0^1(0,\ell_0)} + \|v_0\|_{H_0^1(\ell_0,L)} + |\ell_0 -\ell_T| \leq \delta,
        \]
   	there exist controls $(h_l,h_r)\in L^2(\mathcal{O}_l)\times L^2(\mathcal{O}_r)$ and associated states  $(u,v,\ell)$ with
        \[
   	\left\{
   		\begin{array}{l}
     			\ell\in H^1(0,T)\cap C^{1}((0,T]), \ \ell(t) \in (\ell_l,\ell_r)~~\forall\,\, t\in [0,T],\\
     			\noalign{\smallskip}\displaystyle	
    			u,\,u_x,\,u_t,\,u_{xx}\in L^2(Q_l)~\text{and}~ v,\,v_x,\,v_t,\,v_{xx}\in L^2(Q_r),
   		\end{array}
  	 \right.
        \]
such that 
        \begin{equation}\label{am2}
 \ell(T) = \ell_T,\,\, u(\cdot,T) = 0~~\mbox{in}~~(0,\ell_T)~~\mbox{and}~~
	 v(\cdot,T) = 0~~\mbox{in}~~(\ell_T,L).
        \end{equation}
\end{thm}

\begin{remark}\label{rmq:1}{\rm We will see in Section \ref{sec:lack} that  the maximum principle for parabolic equations implies that the null controllability for \eqref{am} does not hold if one of the controls (for instance $h_r$) vanishes and the interface satisfies $0 < \ell(T) < L$.
	However, the possibility of getting a null control result with only one control when one of the phases collapses to the boundary, that is, $\ell(T) = L$ or $\ell(T) = 0$, is open.
	\Fin}
\end{remark} 
  
	For completeness, let us mention some previous works on the control of our main system and other similar models.
	
	The analysis of the controllability properties for linear and non-linear parabolic PDEs defined in cylindrical domains is a classical problem in control theory and the some 
	of the main contributions are in the references~\cite{fabre, fattorini, cost, ymanuvilov, lebeau}. On the other hand, the study of the controllability properties of free-boundary 
	problems for PDEs has not been much explored, although some important results have been obtained in the last years, specially for one-phase Stefan problems and variants;
	see~\cite{demarque, menezes, tributino1, limaco}. In \cite{borjan}, the authors study the controllability problem for free boundary viscous Burgers equation with one end moving point.
	
	In what respects the two-phase Stefan problem, the best result to our knowledge concerns {\it stabilization}.
	More precisely, it is proved in~\cite{koga} that, under some assumptions, there exist Neumann boundary controls and associated states $(u,v,\ell)$ defined for all~$t > 0$ such that 
        \[
	\lim_{t\rightarrow\infty}\|u(\cdot,t) - \mathcal{T}_m\|_{L^2(0,\ell(t))} = \lim_{t\rightarrow\infty}\|v(\cdot,t) - \mathcal{T}_m\|_{L^2(0,\ell(t))} = 0
	\quad\mbox{and}\quad \lim_{t\rightarrow\infty}\ell(t) = \ell_T ,
        \]
where $\mathcal{T}_m$ is a melting/solidification temperature.

	A natural question is whether or not it is possible to drive both the temperature and the interface to prescribed targets at a finite time.
	In this paper we give a positive partial answer to this question. 
	Recall that, in~\cite{fluid, tributino2, liu}, a similar problem was considered for a 1D fluid-structure problem, with the following equations on the interface:
        \[
      u(\ell(t),t) =v(\ell(t),t) = \ell'(t),\quad v_x(\ell(t),t)-u_x(\ell(t),t) = m\ell''(t)\quad\mbox{for}\quad t\in (0,T).
        \]
	
	In this paper, we deal with situations leading to new difficulties compared to previous works on free-boundary controllability.
	Let us discuss some of these differences:
	
\begin{itemize}

	\item Control of the interface.
	Here, we are also going to control the temperature and also the interface between liquid and solid regions.
	This will bring an extra difficulty.
	They main strategy will rely on linearization, then reformulation as an observability problem with a linear constraint and then resolution of a fixed-point equation.  

	\item Existence of two phases.
	Obviously, this complicates a lot the structure and properties of the state and requires an appropriate analysis.
	
\end{itemize}

	The rest of this paper is organized as follows.
	In~Section~\ref{reformulation}, we will reformulaqte the free-boundary problem as a nonlinear parabolic system in a cylindrical domain.
	In~Section~\ref{sec:positive}, we will present an improved observability inequality which leads to the null controllability for a related linearized system subject to a linear constraint.
	In Section~\ref{sec:nonlinear}, we will give a proof of~Theorem~\ref{twophase}.
	To this purpose, we will apply a fixed-point argument.
	Finally, in~Section~\ref{sec:comments}, we will present some additional comments and questions.


\section{Preliminaries}\label{preliminaries}

	
\subsection{Reformulation of the free-boundary problem} \label{reformulation}

	As a first step, let us find a suitable diffeomorphism $\Phi$ that transforms the free-boundary problem for the parabolic system~\eqref{am} into an equivalent problem for a nonlinear parabolic system in a cylindrical domain.

	To do that, let us fix a function $\ell\in H^1(0,T)\cap C^{1}((0,T])$ such that $\ell(t) \in (\ell_l,\ell_r)$ for all $t\in [0,T]$ and let us take $\sigma > 0$ sufficiently small such that
        \[
\ell_l + \sigma < \ell(t) - \sigma \ \text{ and } \ \ell(t) + \sigma < \ell_r-\sigma \ \text{ in } \ [0,T].
        \]
	Then, for any $\ell_l+2\sigma < y < \ell_r- 2\sigma$, we build a function $m(\cdot,y): \mathbb{R}\to \mathbb{R}$ by linear interpolation of the points $(\ell_l - \sigma,\ell_l - \sigma)$, $(\ell_l + \sigma,\ell_l + \sigma)$, $(y - \sigma, \ell_0- \sigma)$, $(y + \sigma, \ell_0+ \sigma)$, $(\ell_r - \sigma, \ell_r - \sigma)$ and $(\ell_r + \sigma,\ell_r + \sigma)$ and then extend the extreme segments toward infinity.
	Specifically, we have the following definition for $m(\cdot,y)$:
        \[
	m(x,y) :=
		\left\{
   			\begin{array}{lll}
   				x, 														&\mbox{if}&	x\leq \ell_l + \sigma,\\
   				\ell_l + \sigma + \dfrac{(\ell_l - \ell_0 + 2\sigma)(x-\ell_l- \sigma)}{\ell_l + 2\sigma - y}, 	&\mbox{if}&	\ell_l + \sigma < x < y - \sigma,\\
   				x - y + \ell_0, 												&\mbox{if}&	y - \sigma < x < y + \sigma,\\
   				\ell_0 + \sigma + \dfrac{(\ell_r - \ell_0 - 2\sigma)(x - y - \sigma)}{\ell_r - y - 2\sigma}, 	&\mbox{if}&	y + \sigma < x < \ell_r - \sigma,\\
   				x, 														&\mbox{if}&	x \geq \ell_r - \sigma.
   			\end{array}
		\right.
        \]

      	Let us now consider a function $\eta\in C^{\infty}(\mathbb{R})$ such that
        \[
	 \mbox{supp}~\eta \subset (-\sigma,\sigma)~~\int_{-\sigma}^{\sigma}\eta (x)\,d\,x = 1 ~~\text{and}~~\eta(x) =~\eta(-x)
	 \ \  \forall x \in \mathbb{R}.
        \]    
	Then, we can define a smooth function $G : \mathbb{R}\times (\ell_l + 2\sigma, \ell_r - 2\sigma)\mapsto\mathbb{R}$ as follows:
        \[
	G(x,y) := [\eta*m(\cdot,y)](x).
        \]
	A simple computation leads to the equalities
        \begin{equation}\label{am3}
	G(y,y) = \ell_0,\quad \partial_x G(y,y) = 1
        \end{equation}
     and
        \[
	\nabla G(x,y) = \left([\eta'*m(\cdot,y)](x),[\eta*\partial_{y}m(\cdot,y)](x)\right),
        \]
     where
        \[
\partial_{y}m(x,y) :=
	\left\{
   		\begin{array}{lll}
  			0, 													&\mbox{if} & x\leq \ell_l+ \sigma,\\
    			\dfrac{(\ell_l - \ell_0 + 2\sigma)(x-\ell_l- \sigma)}{(\ell_l + 2\sigma - y)^2}, 		&\mbox{if}& \ell_l + \sigma < x < y - \sigma,\\
   			-1, 													&\mbox{if} & y - \sigma < x < y + \sigma,\\
  			\dfrac{(\ell_r - \ell_0 - 2\sigma)(x - \ell_r + \sigma)}{(\ell_r- y - 2\sigma)^2}, 	&\mbox{if} & y + \sigma < x < \ell_r - \sigma,\\
   			0, 													&\mbox{if} & x \geq \ell_r - \sigma.
   		\end{array}
   	\right.
        \]
   	Let us introduce the mapping 
        \[
    	\Phi : Q \mapsto Q, \ \text{ with } \ \Phi(x,t) := \left(G(x,\ell(t)),t\right).
        \]
    	It can be seen that $\Phi$ is a diffeomorphism in~$Q$, coincides with the identity in the regions $(0,\ell_l+\sigma)\times (0,T)$ and $(\ell_r-\sigma,L)\times (0,T)$ and, moreover, $\Phi(\ell(t),t) = (L_0,t)$ for all $t \in [0,T]$. 
	Let us introduce the sets $Q_{0,l} := (0,\ell_0)\times (0,T)$ and $Q_{0,r} := (\ell_0,L)\times (0,T)$ and let us define $p : Q_{0,l} \to \mathbb{R}$ and $q : Q_{0,r} \to \mathbb{R}$, with
        \[
    	p(\xi,t) := u(x,t)=u( \Phi^{-1}(\xi,t))\,\,\,\mbox{and}\,\,\, q(\xi,t) := v(x,t)=v( \Phi^{-1}(\xi,t)),
        \]
    	where $(\xi,t) := \Phi(x,t)$.
	Then, we have that the couple $(p,q)$ satisfies:
        \begin{equation}\label{am4}
    	\left\{
    		\begin{array}{lll}
    			p_t - d^\ell_lp_{\xi\xi} + b^\ell_lp_{\xi} = h_l1_{\mathcal{\omega}_l} 			&\mbox{in} & Q_{0,l},\\
    			q_t - d^\ell_rq_{\xi\xi} + b^\ell_rq_{\xi} = h_r1_{\mathcal{\omega}_r} 			&\mbox{in} & Q_{0,r},\\
    			p(0,\cdot) =  q(L,\cdot) = 0 			&\mbox{on} & (0,T),\\
    			p(\cdot,0) = p_0 										&\mbox{in} & (0,\ell_0),\\
    			q(\cdot,0) = q_0 										&\mbox{in} & (\ell_0,L),\\
    			p(\ell_0,\cdot) = q(\ell_0,\cdot) = 0       &\mbox{on} & (0,T),\\
    			d_l p_{\xi}(\ell_0,t) - d_r q_{\xi}(\ell_0,t) = -\ell'(t) & \mbox{on} & (0,T).
    		\end{array}
    	\right.
        \end{equation}
    	where $p_0 = u_0\circ [G(\cdot,\ell_0)]^{-1}\in H_0^1(0,\ell_0)$, $q_0 = v_0\circ [G(\cdot,\ell_0)]^{-1}\in H_0^1(\ell_0,L)$ and
        \begin{equation}\label{coefficients_diffeom}
        \begin{alignedat}{2}    
    	&d^{\ell}_l (\cdot,t) := d_l\left(G_x\left([G(\cdot,\ell(t))]^{-1},\ell(t)\right)\right)^2,\\
	&d^{\ell}_r (\cdot,t) := d_r\left(G_x\left([G(\cdot,\ell(t))]^{-1},\ell(t)\right)\right)^2,\\
    	&b^{\ell}_l(\cdot,t) := G_y\left([G(\cdot,\ell(t))]^{-1},\ell(t)\right)\ell'(t) + d_lG_{xx}\left([G(\cdot,\ell(t))]^{-1},\ell(t)\right),\\
	&b^{\ell}_r(\cdot,t) := G_y\left([G(\cdot,\ell(t))]^{-1},\ell(t)\right)\ell'(t) + d_rG_{xx}\left([G(\cdot,\ell(t))]^{-1},\ell(t)\right).
\end{alignedat}
        \end{equation}
 \begin{rmk}\label{remark_appen_free}
        {\rm Since $\ell\in H^1(0,T)$, it is not difficult to deduce that $(d^\ell_l,d^\ell_r)\in L^{\infty}(Q_{0,l})\times  L^{\infty}(Q_{0,r})$ and $(b^\ell_l,b^\ell_r)\in L^2(0,T;L^{\infty}(0,\ell_0))\times  L^2(0,T;L^{\infty}(\ell_0,L))$. Moreover, there exist constants $K_1, K_2 > 0$, independent of $\ell$ and $T$, such that $\|(b^\ell_l,b^\ell_r)\|_{L^2(L^{\infty})\times L^2(L^{\infty})} \leq K_1\|\ell'\|_{L^2(0,T)} + K_2 T$.
        \Fin}
        \end{rmk}   
   
    	This way, we have that Theorem~\ref{twophase} is equivalent to prove a local controllability result for~\eqref{am4}.
	Actually, we will prove the following:
	
\begin{thm}\label{twophase1}
    	Let $\ell_T\in (\ell_l,\ell_r)$. Then there exists $\delta > 0$ such that, for any $p_0\in H_0^1(0,\ell_0)$ with $p_0 \geq 0$, any $q_0\in H_0^1(\ell_0,L)$ with $q_0 \leq 0$ and any 
	$\ell_0\in (\ell_l,\ell_r)$ satisfying 
        \[
	\|p_0\|_{H_0^1(0,\ell_0)} + \|q_0\|_{H_0^1(\ell_0,L)} + |\ell_0 -\ell_T|< \delta,
        \]
there exist controls $(h_l,h_r)\in L^2(\mathcal{O}_l)\times L^2(\mathcal{O}_r)$ and associated  solutions $(p,q,\ell)$ to  \eqref{am4} with
        \[
   	\left\{
   		\begin{array}{l}
     			\ell\in H^1(0,T)\cap C^{1}((0,T]), \ \ \ell(t) \in (\ell_l,\ell_r)\ \ \forall t\in [0,T], \\
     			\noalign{\smallskip}\displaystyle	
			p,p_{\xi},p_{t},p_{\xi\xi} \in L^2(Q_{0,l})~~\mbox{and}~~q,q_{\xi},q_{t},q_{\xi\xi} \in L^2(Q_{0,r}),\\
	\end{array}
  	 \right.
        \]
such that 
        \[
      \ell(T) = \ell_T,~~p(\cdot,T) = 0~~\mbox{in}~~(0,\ell_0)~~\mbox{and}~~q(\cdot,T) = 0~~\mbox{in}~~(\ell_0,L).   
        \]
\end{thm}


\subsection{Well-posedness of the two-phase free-boundary problem}

	The aim of this section is to prove the local existence and uniqueness for the two-phase free-boundary problem \eqref{am}.
	More precisely, we have the following result:
 
\begin{prop} \label{appendix_thm_free}
	Let $L, T > 0$ and $\ell_l < \ell_0 < \ell_r$ be given.
	Then, the system \eqref{am} is locally well-posed. 
	In other words, for any $(h_l,h_r) \in L^{2}(\mathcal{O}_l)\times L^2(\mathcal{O}_r)$ and $(u_0,v_0) \in H_0^1(0,\ell_0)\times H_0^1(\ell_0,L)$, there exist a 
	time $0 < \widehat{T} \leq T$  and a unique strong solution to~\eqref{am} in the time interval $(0,\widehat{T})$ such that
	\[
\left\{
	\begin{array}{l}
        		\ell\in H^{1}(0,\widehat{T}), \ \ \ell(0) = \ell_0, \ \ \ell(t) \in (\ell_l,\ell_r) \ \ \forall t \in [0,\widehat{T}], \\
        		\noalign{\smallskip}\displaystyle	
        		u,\,u_x,\,u_t,\,u_{xx}\in L^2(\widehat{Q}_{l})~\text{and}~ v,\,v_x,\,v_t,\,v_{xx}\in L^2(\widehat{Q}_{r}), \\
   	\end{array}
\right.
	\]
	where~$\widehat{Q}_l :=\{(x,t)\in Q: t\in (0,\widehat{T}),\, x\in (0,\ell(t))\}$~and~$\widehat{Q}_r :=\{(x,t)\in Q: t\in (0,\widehat{T}),\,x\in (\ell(t),L)\}$.
\end{prop} 

 	Thanks to the diffeomorphism $\Phi : Q \mapsto Q$, introduced in Section \ref{reformulation}, Proposition \ref{appendix_thm_free} is equivalent to the local existence and uniqueness of~\eqref{am4}.
	More precisely, Proposition~\ref{appendix_thm_free} is an immediate consequence of the following result:

\begin{prop}\label{appendix_thm_free2}
	Let the conditions of~Proposition~\ref{appendix_thm_free} be satisfied.
	Then, the nonlinear system~\eqref{am4} is locally well-posed, i.e.~if $(h_l,h_r) \in L^{2}(\mathcal{O}_l)\times L^2(\mathcal{O}_r)$ and~$(p_0,q_0) \in H_0^1(0,\ell_0)\times H_0^1(\ell_0,L)$ are fixed, there exist $\widehat{T} \in (0,T)$  and a unique strong solution in~$(0,\widehat{T})$ such that
\[
\left\{
   	    \begin{array}{l}
        \ell\in H^{1}(0,\widehat{T}),~\ell(0) = \ell_0,~ \ \ell(t) \in (\ell_l,\ell_r)~~\forall\,\, t\in [0,\widehat{T}],\\
        \noalign{\smallskip}\displaystyle	
        p,\,p_{\xi},\,p_t,\,p_{\xi\xi}\in L^2(\widehat{Q}_{0,l})~\text{and}~ q,\,q_{\xi},\,q_t,\,q_{\xi\xi}\in L^2(\widehat{Q}_{0,r}),
   		\end{array}
\right.
\]
	where $\widehat{Q}_{0,l} := (0,\ell_0) \times (0,\widehat{T})$~and~$\widehat{Q}_{0,r} := (\ell_0,L) \times (0,\widehat{T})$. 
\end{prop}
\begin{proof}
 	First, let us introduce the spaces
\[
	\begin{array}{l} \displaystyle
  	X^{T}_l \!:=\! L^2(0,\! T;\! H^2(0,\! \ell_0)) \!\cap\! H^1(0,\! T;\! L^2(0,\! \ell_0)) \ \text{ and } \
	X^{T}_r \!:=\! L^2(0,\! T;\! H^2(\ell_0,\! L)) \!\cap\! H^1(0,\! T;\! L^2(\ell_0,\! L)).
	\end{array}
\]

   For each fixed $\ell \in H^1(0,T)$, we can use the Faedo-Galerkin method to get a unique  strong solution~$(p^{\ell},q^{\ell}) \in X^{T}_l \times X^{T}_r$ to~\eqref{am4}$_1$--\eqref{am4}$_6$.
   Moreover, thanks to Remark~\ref{remark_appen_free}, we find a positive constant $C_1$, independent of $p_0, q_0, h_l, h_r, \ell$ and $T$, 
   such that
\begin{equation}\label{estimates_appendix_free}
	\|(p,q)\|^2_{X^{T}_l\times X^{T}_r} \leq C_1\left(1+K+\sqrt{T}\right)\left[1+\left(K+ \sqrt{T}\right)K + T\right]
	e^{C_1\left(1 +  \sqrt{T}\right)\left(K +  \sqrt{T}\right)}\Pi^2,
\end{equation}
	where $K := \|\ell'\|_{L^2(0,T)}$ and $\Pi:= \|(p_0,q_0)\|_{H_0^1\times H_0^1} +  \|(h_l,h_r)\|_{L^2\times L^2}$.
    
   	Now, let us assume that $\ell_l < \widehat{\ell}_l <\ell_0< \widehat{\ell}_r < \ell_r$ and $R > 0$ and let us introduce the set
 $$
     \mathcal{A}_{R,T}:= \{\ell\in H^{1}(0,T): \widehat{\ell}_l \leq \ell(t) \leq \widehat{\ell}_r,~\forall t\in [0,T],~~\ell(0) = \ell_0,~ \|\ell'\|_{L^2(0,T)} \leq R\}
  $$     
and the mapping $\Lambda : \mathcal{A}_{R,T} \mapsto H^1(0,T)$, with
  \begin{equation*}\label{fixedpointmap}
	      \Lambda(\ell)= \mathcal{L}_\ell,~~\hbox{and}~~\mathcal{L}_\ell(t) := \ell_0 - \displaystyle\int_{0}^{t}\left[d_l p^{\ell}_{\xi}(\ell_0, \tau ) - d_r q^{\ell}_{\xi}(\ell_0, \tau )\right]\,d\tau,
\end{equation*}
	where $(p^{\ell},q^{\ell}) \in X^{T}_l\times X^{T}_r$ is the unique strong solution to \eqref{am4}$_1$--\eqref{am4}$_6$.
	It is not difficult to see that $\mathcal{A}_{R,T}$ is a non-empty, closed and convex subset of $H^{1}(0,T)$.
	
	Let us check that, for some $0 < \widehat{T} \leq T$, $\Lambda$ satisfies the assumptions of Banach's Fixed-Point Theorem in~$\mathcal{A}_{R,\widehat{T}}$:
	
\begin{itemize}
	  \item {\it There exists $\widetilde{T} \in (0,T]$ such that
	  \begin{equation} \label{new-tau}
\Lambda(\mathcal{A}_{R,\tau}) \subset\mathcal{A}_{R,\tau} \quad \forall \tau \in (0,\widetilde{T}].
	  \end{equation}}
	  
	  Indeed, $ \mathcal{L}_\ell(0) = \ell_0$.
	  Let us introduce
\begin{equation}\label{estimate_appendix_fixedpoint_1}
        C(T,R):=C_1\left(1+R+\sqrt{{T}}\right)\left[1+\left(R+ \sqrt{{T}}\right)R + {T}\right]
	e^{C_1\left(1 +  \sqrt{{T}}\right)\left(R +  \sqrt{{T}}\right)}.
\end{equation}       
         Then, using the H\"older inequality, \eqref{estimates_appendix_free} and~\eqref{estimate_appendix_fixedpoint_1}, we see that, for some~$C_2 > 0$ independent of $T$, one has:
\begin{align*}
  |\mathcal{L}_\ell(t) - \ell_0| & \leq \int_{0}^{t}|d_lp^{\ell}_{\xi}(\ell_0,\tau) - d_rq^{\ell}_{\xi}(\ell_0,\tau)|\,d\,\tau\\
                & \leq C_2\widetilde{T}^{1/2}
                   \|(p_{\xi}^{\ell}(\ell_0,\cdot),q_{\xi}^{\ell}(\ell_0,\cdot))\|_{L^2(0,\widetilde{T})\times L^2(0,\widetilde{T}) }\\
                & \leq C_2 \widetilde{T}^{1/2} \|(p^{\ell},q^{\ell})\|_{X^{{T}}_l\times X^{{T}}_r}\\
                & \leq C_2 \widetilde{T}^{1/2}C(T,R)^{1/2} \, \Pi
\end{align*}
for all $t \in [0,\widetilde{T}]$ and any~$\widetilde{T} \in (0,T]$.

  	On the other hand, since the trace operator $\Gamma:X^T_l\times X^T_r\mapsto H^{1/4}(0,T)\times H^{1/4}(0,T)$ defined by~$\Gamma(p,q) := (p^\ell_\xi(\ell_0,\cdot),q^\ell_\xi(\ell_0,\cdot))$ is continuous, thanks to the continuity of the embedding $H^{1/4}(0,T) \hookrightarrow L^4(0,T)$, we find~$C_3 > 0$ (independent of $T$) such that 
 $$
 	\|\mathcal{L}_\ell^\prime\|_{L^2(0,\widetilde{T})} \leq C_3\widetilde{T}^{1/4}C(T,R)^{1/2} \,\Pi.
$$ 
	It follows easily from the inequalities above that, if~$\widetilde{T}$ is sufficiently small, \eqref{new-tau} holds.

\item {\it There exists $\widehat{T} \in (0,T]$ such that $\Lambda : \mathcal{A}_{R,\widehat{T}} \mapsto \mathcal{A}_{R,\widehat{T}}$ is a contraction.}

   	It can be proved that there exists~$D_1 = D_1(\tau)$ such that $D_1(\tau) \rightarrow 0 $ as~$\tau \rightarrow 0$ and
     \begin{equation}\label{uniqueness_freeboundary_1}
\|\mathcal{L}_{\ell_1} -\mathcal{L}_{\ell_2}\|_{H^1(0,\tau)} \leq D_1(\tau)\left(\|p^{\ell_1} - p^{\ell_2} \|_{X^{\tau}_l} + \|q^{\ell_1} - q^{\ell_2} \|_{X^{\tau}_r}\right) \quad \forall \ell_1,\ell_2 \in \mathcal{A}_{R,\tau}.
     \end{equation}
     
    Furthermore, using standard energy estimates, we get a positive~$D_2 = D_2(\tau)$, similar to~\eqref{estimate_appendix_fixedpoint_1}, such that 
     \begin{equation}\label{uniqueness_freeboundary_2}
    \|p^{\ell_1} - p^{\ell_2} \|_{X^{\tau}_l} + \|q^{\ell_1} - q^{\ell_2} \|_{X^{\tau}_r} \leq D_2(\tau)\|(F^{\ell}_l,F^{\ell}_r)\|_{L^2(\widetilde{Q}_{0,l})\times L^2(\widetilde{Q}_{0,r})},
     \end{equation}
     where  
    \begin{equation}\label{uniqueness_freeboundary}
     	F^{\ell}_{l} :=  (d^{\ell_1}_l - d^{\ell_2}_l)p^{\ell_2}_{\xi\xi} - (b^{\ell_1}_l - b_l^{\ell_2})p^{\ell_2}_{\xi}
     	~~\mbox{and}~~
	F^{\ell}_{r} :=  (d^{\ell_1}_r - d^{\ell_2}_r)q^{\ell_2}_{\xi\xi} - (b^{\ell_1}_r - b_r^{\ell_2})q^{\ell_2}_{\xi}
    \end{equation} 
    	and $\widetilde{Q}_{0,l} := (0,\ell_0) \times (0,\tau)$~and~$\widetilde{Q}_{0,r} := (\ell_0,L) \times (0,\tau)$.
     
     Then, using the fact that~$G$ and its inverse $G^{-1}$
     (defined in~Section~\ref{reformulation}) are smooth functions, we see that there exists $D_3 = D_3(\tau)$
     ($D_3(s)$ is bounded for $0\leq s\leq T$)  such that 
     \begin{equation}\label{uniqueness_freeboundary_3}
\|(F^{\ell}_l,F^{\ell}_r)\|_{L^2(\widetilde{Q}_{0,l})\times L^2(\widetilde{Q}_{0,r})} \leq D_3(\tau)\|(p^{\ell_2},q^{\ell_2})\|_{X^{\tau}_l\times X^{\tau}_r}\|\ell_1 - \ell_2\|_{H^1(0,\tau)}. 
      \end{equation}
      Combining \eqref{uniqueness_freeboundary_1}--\eqref{uniqueness_freeboundary_3}, we deduce that
\begin{equation}\label{uniqueness_freeboundary_4}
	\|\Lambda(\ell_1) - \Lambda(\ell_2)\|_{H^1(0,\tau)}
	\leq E(\tau)\|\ell_1 - \ell_2\|_{H^1(0,\tau)}, 
\end{equation}
     where $E(\tau) := D_1(\tau)D_2(\tau)D_3(\tau)\|(p^{\ell_2},q^{\ell_2})\|_{X^{\tau}_l\times X^{\tau}_r}$.
     Since $(p^{\ell_2},q^{\ell_2})$ is uniformly bounded in~$X^{\tau}_l\times X^{\tau}_r$ for all $0<\tau\leq \widetilde{T}$ provided~$\ell_2 \in \mathcal{A}_{R,\tau}$, we find that $E(s) \rightarrow 0 $ as~$s \rightarrow 0$.
     
     As a consequence, there exists $\widehat{T} \in (0,\widetilde{T}]$ such that $\Lambda : \mathcal{A}_{R,\widehat{T}} \mapsto \mathcal{A}_{R,\widehat{T}}$ is a contraction.
%
%
%
%
%
     \end{itemize}
 
     As a consequence, $\Lambda$ possesses exactly one fixed-point in~$\mathcal{A}_{R,\widehat{T}}$.
     This ends the proof.
\end{proof}

  \section{Approximate controllability of a linearized system}\label{sec:positive} 
   
 	In this section, we are going to complete a first step in the proof of Theorem \ref{twophase1}. More precisely, we are going to prove a controllability result for a suitable (natural) 
	linearization of  \eqref{am4}.
 
	To do this, let us fix $\ell\in C^{1}([0,T])$ with $\ell(0) = \ell_0$ and $\ell([0,T])\subset (\ell_l,\ell_r)$ and let us consider the system: 
        \begin{equation}\label{am7}
      	\left\{
      		\begin{array}{lll}
      			M^{\ell}_{l}(p) = h_l1_{\mathcal{\omega}_l} 					&\mbox{in}&	Q_{0,l},\\
      			M^{\ell}_{r}(q) = h_r1_{\mathcal{\omega}_r} 					&\mbox{in}& 	Q_{0,r},\\
      			p(0,\cdot) = p(\ell_0,\cdot) = q(\ell_0,\cdot) = q(L,\cdot) = 0 	&\mbox{on}& 	(0,T),\\
      			p(\cdot,0) = p_0 								&\mbox{in}& 	(0,\ell_0),\\
      			q(\cdot,0) = q_0 								&\mbox{in}& 	(\ell_0,L),
      		\end{array}
      	\right.
        \end{equation}
	where the operators $M^{\ell}_{l}$ and $M^{\ell}_{r}$  are respectively defined by:
           \[
     	M^{\ell}_{l}(p) := p_t - d^{\ell}_lp_{\xi\xi} + b^{\ell}_lp_{\xi}\quad\mbox{and}\quad M^{\ell}_{r}(q) := q_t - d^{\ell}_rq_{\xi\xi} + b^{\ell}_rq_{\xi}.
        \]
	Also, let us introduce the function $\mathcal{L} : [0,T] \mapsto \mathbb{R}$ given by
        \[
       	\mathcal{L}(t) := \ell_0 - \int_{0}^{t}\left[d_l p_{\xi}(\ell_0, \tau ) - d_r q_{\xi}(\ell_0, \tau )\right]d \tau. 
        \]
        
\begin{remark}\label{rmk:regul_coef}
   	{\rm Using the fact that $G$ and $G^{-1}$ are smooth and~$\ell\in C^1([0,T])$, we can prove that~$d^{\ell}_{l}$ and~$b^{\ell}_{l}$ belong, respectively, to~$C^{1}(\overline{Q}_{0,l})$ and~$C^{0}(\overline{Q}_{0,l})$.
	Furthermore, the second spatial derivative of $d^{\ell}_{l}$ and the first spatial derivative of $b^{\ell}_{l}$ are functions in $C^{0}(\overline{Q}_{0,l})$.
	The same can be obtained for the coefficients $d^{\ell}_{r}$ and $b^{\ell}_{r}$.
	\Fin}
\end{remark}

       	The main goal of this section is to obtain a (robust) approximate controllability result for  \eqref{am7} subject to the linear constraint 
        \begin{equation}\label{am8}
       	\mathcal{L}(T) = \ell_T.
        \end{equation}
     	In other words, we want to find controls $(h_l,h_r)\in  L^2(\mathcal{O}_l)\times L^2(\mathcal{O}_r)$ 
	such that the associated solutions to~\eqref{am7} satisfy~\eqref{am8}. 
	
	Let us first reformulate~\eqref{am8}. 
	Thus, consider the augmented adjoint system
        \begin{equation}\label{am9}
       	\left\{
       		\begin{array}{lll}
       			(M_{l}^\ell)^{*}(\psi) = 0 					&\mbox{in}& 	Q_{0,l},\\
       			(M_{r}^\ell)^{*}(\zeta) = 0 					&\mbox{in} & 	Q_{0,r},\\
       			\psi(0,\cdot) = 0, \quad \psi(\ell_0,\cdot) = 1 		&\mbox{on} & 	(0,T),\\
      	 		\zeta(\ell_0,\cdot) = 1,\quad \zeta(L,\cdot) = 0, 	&\mbox{on} & 	(0,T),\\
       			\psi(\cdot,T) = 0 						&\mbox{in} & 	(0,\ell_0),\\
       			\zeta(\cdot,T) = 0 						&\mbox{in} & 	(\ell_0,L).
       		\end{array}
      	\right.
        \end{equation}
        It is not difficult to check that~\eqref{am9} possesses a unique weak solution $(\psi_\ell, \zeta_\ell)$, with
        	        \[
      		\psi_{\ell}  \in L^2(0,T;H_0^1(0,\ell_0))\cap H^1(0,T;H^{-1}(0,\ell_0)),
	        \]
	        \[
		\zeta_{\ell}\in L^2(0,T;H_0^1(\ell_0,L))\cap H^1(0,T;H^{-1}(\ell_0,L)).
	        \]
	
	A crucial property of~$(\psi_{\ell},\zeta_{\ell})$ is the following:
	
\begin{prop}\label{proplimitation_coef}
	Given $R > 0$, let us consider the set $\mathcal{B}_R := \{\ell\in C^1([0,T]);~~\|\ell'\|_{C^0([0,T])} \leq R\}$. Then, there exists a positive constant $C_0$, only depending on~$\ell_0$, $\ell_l$, $\ell_r$, $\omega_l$, $\omega_r$, $T$ and $R$ such that, for any~$\ell \in \mathcal{B}_R$, one has:
        \[
      \|\psi_{\ell}\|_{L^2(\mathcal{O}_l)} + \|\zeta_{\ell}\|_{L^2(\mathcal{O}_r)} \geq C_0 .
        \]
\end{prop}

\begin{proof}
	We argue by contradiction.
	Thus, if the assertion were false, then there would exist~$\ell_1, \ell_2,\ldots$ and associated pairs~$(\psi^{1},\zeta^{1}),(\psi^{2},\zeta^{2}),\ldots$ (weak solutions to~\eqref{am9}), such that 
        \begin{equation}\label{ampsi}
	\|\ell'_n\|_{\infty} \leq R~~\mbox{and}~~\|\psi^n\|_{L^2(\mathcal{O}_l)} + \|\zeta^n\|_{L^2(\mathcal{O}_r)} \leq \dfrac{1}
	{n}~~\forall~n \geq 1.
        \end{equation}
	Due to the smoothing effect of parabolic operators and the fact that the $(d^{\ell_n}_{l},b^{\ell_n}_{l})$ are uniformly bounded in~$C^{1}(\overline{Q}_{0,l})\times C^{0}(\overline{Q}_{0,l})$, there would exist $\sigma > 0$ such that
        \[
	\|\psi^{n}\|_{L^2(0,T-\sigma; H^2(0,\ell_0))} + \|\psi_t^{n}\|_{L^2(0,T-\sigma; L^{2}(0,\ell_0))} \leq 
	C~~\forall~n \geq 1,
        \]
with $C > 0$ depending only on $\ell_0$, $\ell_l$, $\ell_r$, $T$ and~$R$.
	Consequently, after extraction of a subsequence, we would have
        \[
        \left\{
        \begin{array}{lll}
         \ell_n \rightarrow \ell & \mbox{strongly in} & C^0([0, T-\sigma]),\\
         \ell_n \rightarrow \ell & \mbox{weakly in} & H^1(0,T-\sigma),\\
         \psi^{n} \rightarrow \psi & \mbox{weakly in} & L^2(0,T-\sigma; H^2(0,\ell_0)) \cap H^1(0,T-\sigma; L^{2}(0,\ell_0))
        \end{array}
         \right.
        \] 
and we would be able to pass to the limit in the equation and the boundary condition satisfied by $\psi^{n}$ to deduce that
        \begin{equation}\label{psiequation_free}
          \left\{
          \begin{array}{lll}
          (M^{\ell}_{l})^{*}(\psi) = 0 & \mbox{in} & (0,\ell_0)\times (0,T-\sigma),\\
          \psi(0,\cdot) = 0,~ \psi(\ell_0,\cdot) = 1 & \mbox{on} &  (0,T-\sigma).
          \end{array}
          \right.
        \end{equation}
But we would also have, by~\eqref{ampsi}, that $\psi \equiv 0$ in~$\omega_{l}\times (0,T-\sigma)$, which is impossible, in view of the unique continuation property and \eqref{psiequation_free}$_2$. This ends the proof.
\end{proof}

\

	Let us multiply $\eqref{am7}_1$ by $\psi_{\ell}$ and let us integrate in $Q_{0,l}$ to obtain     
\begin{align}\label{am10}
     	\iint_{\mathcal{O}_{l}}h_l\psi_{\ell}\,d\xi\,dt = -\int_{0}^{\ell_0}p_0(\xi)\psi_{\ell}(\xi,0)\,d\xi - \int_{0}^{T}d_l p_{\xi}
	(\ell_0,\tau)\,d\tau.
\end{align}       
       	Analogously, multiplying $\eqref{am7}_2$ by $\zeta_{\ell}$ and integrating in $Q_{0,r}$ we get
\begin{align}\label{am11}
   	\iint_{\mathcal{O}_r}h_r\zeta_{\ell}\,d\xi\,dt =  -\int_{\ell_0}^{L}q_0(\xi)\zeta_{\ell}(\xi,0)\,d\xi + \int_{0}^{T}d_r q_{\xi}
	(\ell_0, \tau )\,d \tau .
\end{align}
       It follows from \eqref{am10}--\eqref{am11}  that a couple of controls $(h_l,h_r)\in L^2(\mathcal{O}_l)\times L^2(\mathcal{O}_r)$ is such that $\mathcal{L}(T) = \ell_T$ if and only if
\begin{align}\label{am12}
	\iint_{\mathcal{O}_l}h_l\psi_{\ell}\,d\xi\,dt + \iint_{\mathcal{O}_r}h_r\zeta_{\ell}\,d\xi\,dt = \ell_T- \ell_0 - \int_{0}^{\ell_0}
	p_0(\xi)\psi_{\ell}(\xi,0)\,d\xi - \int_{\ell_0}^L q_0(\xi)\zeta_{\ell}(\xi,0)\,d\xi.
\end{align}
       
       In Section~\ref{approximatedsection}, we will establish the approximate controllability of \eqref{am7} subject to the linear constraint~\eqref{am12}.
       Before, we will need an adequate (improved) observability inequality.

   
\subsection{An improved observability inequality}

	To do this, let us first consider open sets $\omega_{0,l}\subset\subset \omega_{l}$, $\omega_{0,r}\subset\subset \omega_{l}$ and let us introduce the weight functions
	$\eta_{0,l}\in~C^2([0,\ell_0])$ and $\eta_{0,r}\in C^2([\ell_0,L])$ satisfying
        \[
    	\left\{
       		\begin{array}{l}
      			\eta_{0,l} > 0\,\,\mbox{in}\,\,(0,\ell_0),\,\,\eta_{0,l}(0) = \eta_{0,l}(\ell_0)= 0\,\,\mbox{and}\,\,|\eta_{0,l}'| > 0\,\,\mbox{in}\,\,[0,\ell_0]\backslash \omega_{0,l},\\
       			\eta_{0,r}  > 0\,\,\mbox{in}\,\,(\ell_0,L),\,\,\eta_{0,r}(\ell_0) = \eta_{0,r}(L)= 0\,\,\mbox{and}\,\,|\eta_{0,r}'| > 0\,\,\mbox{in}\,\,[\ell_0,L]\backslash \omega_{0,r}.
       		\end{array}
       \right.
        \]
       Also, for any~$\lambda > 0$, let us set
        \[
    	\left\{
       		\begin{array}{l}
       			\mu_{l}(\xi,t) := \dfrac{e^{\lambda \eta_{0,l}(\xi)} }{t(T - t)},\,\,\,\alpha_{l}(\xi,t) := \dfrac{e^{2\lambda\|\eta_{0,l}\|_{\infty}} -  e^{\lambda \eta_{0,l}(\xi)}}{t(T - t)},\\
       			\\
        			\mu_{r}(\xi,t) := \dfrac{e^{\lambda \eta_{0,r}(\xi)} }{t(T - t)},\,\,\,\alpha_{r}(\xi,t) := \dfrac{e^{2\lambda\|\eta_{0,r}\|_{\infty}} -  e^{\lambda \eta_{0,r}(\xi)}}{t(T - t)}.
       		\end{array}
    	\right.
        \]
    	Then, by using the regularity of the coefficients of the adjoint operators $(M^\ell_l)^{*} $ and $(M^\ell_r)^{*}$ (see Remark \ref{rmk:regul_coef}) and following the ideas in~\cite{guerreroo, ymanuvilov}, we get the following global Carleman estimates:

\begin{prop}
     	Let $R>0$ and assume that $\ell\in C^{1}([0,T])$ satisfies $\ell(0) = \ell_0$, $\ell([0,T])\subset (\ell_l,\ell_r)$ and $\|\ell'\|_{C^0([0,T])}\leq R$. Then, there exist positive constants $\lambda_0,  s _0$ and $C$ (depending on $\ell_l, \ell_r, R,\omega_{l},\omega_{r}$ and $T$) such that, for any $ s \geq s _0$ and
	$\lambda\geq \lambda_0$, we have:
\begin{align}\label{am14}
      	\iint_{Q_{0,l}}e^{-2 s \alpha_{l}}\left[( s \mu_{l})^{-1}\left(|\varphi_t|^2 + |\varphi_{\xi\xi}|^2\right) + \lambda^2( s \mu_{l})|\varphi_{\xi}|^2 + \lambda^4( s \mu_{l})^3|\varphi|^2\right]d\xi\,dt\\
      	\leq C\left[\iint_{Q_{0,l}} e^{-2 s \alpha_{l}}\left|(M^\ell_l)^{*}(\varphi)\right|^2\,d\xi\,dt 
	+\iint_{\mathcal{O}_{l}}e^{-2 s \alpha_l}\lambda^4( s \mu_l)^3|\varphi|^2\,d\xi\,dt\right]\nonumber    
\end{align}
   	and
\begin{align}\label{am15}
       	\iint_{Q_{0,r}}e^{-2 s \alpha_r}\left[( s \mu_r)^{-1}\left(|\phi_t|^2 + |\phi_{\xi\xi}|^2\right) + \lambda^2( s \mu_r)|\phi_{\xi}|^2 + \lambda^4( s \mu_r)^3|\phi|^2\right]d\xi\,dt\\
      	\leq C\left[ \iint_{Q_{0,r}} e^{-2 s \alpha_r}|(M^\ell_r)^{*}(\phi)|^2\,d\xi\,dt + \iint_{\mathcal{O}_r}e^{-2 s \alpha_r}\lambda^4( s \mu_r)^3|\phi|^2\,d\xi\,dt\right],\nonumber    
\end{align}
    	for any pair $(\varphi,\phi)$ in the Bochner-Sobolev space 	
	        \[
		[L^2(0,T;H_0^1(0,\ell_0))\cap  H^1(0,T;H^{-1}(0,\ell_0))] \times [L^2(0,T;H_0^1(\ell_0,L))\cap H^1(0,T;H^{-1}(\ell_0,L))]
	        \] 
	such that $\left((M^\ell_l)^{*}(\varphi),(M^\ell_r)^{*}(\phi)\right)$ belongs to~$L^2(Q_{0,l})\times L^2(Q_{0,r})$.
\end{prop}

      A straightforward argument, based on the estimates \eqref{am14}--\eqref{am15}, leads the 
      following observability inequality:
      
\begin{prop}
    	Let $R>0$ and assume that $\ell\in C^{1}([0,T])$ satisfies $\ell(0) = \ell_0$, $\ell([0,T])\subset (\ell_l,\ell_r)$ and $\|\ell'\|_{C^0([0,T])}\leq R$. There exist positive constants $\lambda_0,  s _0$ and $C$, depending on $\ell_l$, $\ell_r$, $R$,  $\omega_l$, $\omega_r$ and~$T$, such that, for any $s \geq s _0$ and any~$\lambda \geq \lambda_0$, we have:
\begin{align}\label{am16}
      	\|(\varphi(\cdot,0),\phi(\cdot,0))\|_{L^2(0,\ell_0)\times L^2(\ell_0,L)} \leq C\|(\varphi,\phi)\|_{L^2(\mathcal{O}_l)\times L^2(\mathcal{O}_r)},
\end{align}
      for any pair $(\varphi,\phi)$ in the Bochner-Sobolev space 
	        \[
	[L^2(0,T;H_0^1(0,\ell_0))\cap  H^1(0,T;H^{-1}(0,\ell_0))] \times [L^2(0,T;H_0^1(\ell_0,L))\cap H^1(0,T;H^{-1}(\ell_0,L))]
	        \]
such that $\left((M^\ell_l)^{*}(\varphi),(M^\ell_r)^{*}(\phi)\right) = (0,0)$.
\end{prop}

	In order to present an improved observability inequality, let us introduce the linear projectors $\mathbb{P}^{\ell}_l: L^2(Q_{0,l}) \mapsto L^2(Q_{0,l})$ and $\mathbb{P}^{\ell}_r: L^2(Q_{0,r}) \mapsto  L^2(Q_{0,r})$, respectively given by
        \[
      	\mathbb{P}^{\ell}_l{\varphi} := \beta^{\ell}_l(\varphi)\psi_{\ell} \ \text{ and } \ 
      	\mathbb{P}^{\ell}_r{\phi} := \beta^{\ell}_r(\phi)\zeta_{\ell},
        \]
where we have set
        \[
	\beta^{\ell}_l(\varphi) := \frac{\displaystyle\iint_{\mathcal{O}_l}\psi_{\ell}\varphi\,d\xi\,dt}
	{\displaystyle\iint_{\mathcal{O}_l}|\psi_{\ell}|^2 \,d\xi\,dt} \quad \text{and}\quad 
	\beta^{\ell}_r(\phi) := \frac{\displaystyle\iint_{\mathcal{O}_r}\zeta_{\ell}\phi\,d\xi\,dt}
	{\displaystyle\iint_{\mathcal{O}_r}|\zeta_{\ell}|^2\,d\xi\,dt}
        \]
and $(\psi_{\ell},\zeta_{\ell})$ is the unique weak solution to \eqref{am9}.

\begin{remark}
	{\rm Note that the ranges of $\mathbb{P}^{\ell}_l$ and $\mathbb{P}^{\ell}_r$ are 1D vector spaces.
	Therefore, these operators are compact.
	\Fin}
	\end{remark}

	For any $(\varphi_T, \phi_T) \in L^2(0,\ell_0)\times L^2(\ell_0,L)$, there exists an unique couple $(\varphi,\phi)$ satisfying
        \begin{equation}\label{am999}
	\varphi  \in L^2(0,T;H_0^1(0,\ell_0))\cap H^1(0,T;H^{-1}(0,\ell_0)) \ \text{ and } \
	\phi\in L^2(0,T;H_0^1(\ell_0,L))\cap H^1(0,T;H^{-1}(\ell_0,L)),
        \end{equation}
that solves in the weak sense the linear system
        \begin{equation}\label{am1000}
       	\left\{
       		\begin{array}{lll}
       			(M_{l}^\ell)^{*}(\varphi) = 0 					&\mbox{in}& 	Q_{0,l},\\
       			(M_{r}^\ell)^{*}(\phi) = 0 					&\mbox{in} & 	Q_{0,r},\\
       			\varphi(0,\cdot) = \varphi(\ell_0,\cdot) = 0 		&\mbox{on} & 	(0,T),\\
      	 		\phi(0,\cdot) = \phi(\ell_0,\cdot) = 0 	&\mbox{on} & 	(0,T),\\
       			\varphi(\cdot,T) = \varphi_T 						&\mbox{in} & 	(0,\ell_0),\\
       			\phi(\cdot,T) = \phi_T 						&\mbox{in} & 	(\ell_0,L).
       		\end{array}
      	\right.
        \end{equation}	
Accordingly, we can introduce the following functional in $L^2(0,\ell_0)\times L^2(\ell_0,L)$:
\begin{align*}
       	I(\varphi_T,\phi_T) := \iint_{\mathcal{O}_l}|\varphi|^2\,d\xi\,dt + \int_{0}^{\ell_0}|\varphi(\xi,0)|^2\,\,d\xi
	+|\beta^{\ell}_l(\varphi)|^2 \\
        + \iint_{\mathcal{O}_r}|\phi|^2 \,d\xi\,dt + \int_{\ell_0}^{L}|\phi(\xi,0)|^2\,d\xi+ |\beta^{\ell}_r(\phi)|^2,
\end{align*}
where $(\varphi,\phi)$ satisfies \eqref{am999}--\eqref{am1000}.
		
       	We can prove the following result:
\begin{prop}\label{improved}
	Let $R>0$ and let us assume that $\ell\in C^1([0,T])$ satisfies $\ell(0) = \ell_0$, $\ell([0,T])\subset (\ell_l,\ell_r)$ and $\|\ell'\|_{C^0([0,T])}\leq R$.  Then, there exists a positive constant $C$, depending on $\ell_0$, $\ell_l$, $\ell_r$, $R$, $\omega_l$, $\omega_r$ and~$T$, such that, for any 
         $(\varphi_T,\phi_T)\in L^2(0,\ell_0)\times L^2(\ell_0,L) $, the following holds:
        \begin{equation}\label{am18}
\begin{alignedat}{2}
        	I(\varphi_{T},\phi_{T}) \leq &~ C\left[ 
	 \iint_{\mathcal{O}_l} |\varphi - \mathbb{P}^{\ell}_l{\varphi}|^2\,d\xi\,dt
    + \iint_{\mathcal{O}_r} |\phi - \mathbb{P}^{\ell}_r{\phi}|^2\,d\xi\,dt\right].
\end{alignedat}
        \end{equation}
\end{prop}

\begin{proof}
 	The proof will be by contradiction.
	It is inspired by the results in \cite{nakoulima}.
	
 	Let us prove that there exists a constant $C_1>0$ (depending on $\ell_0, \ell_l, \ell_r,R, \omega_l,\omega_r$ and $T$) such that, for any couple of functions $(\varphi_{T},\phi_{T})\in L^2(0,\ell_0)\times L^2(\ell_0,L)$, one has
        \begin{equation}\label{am1800}
	\begin{alignedat}{2}
	&\iint_{\mathcal{O}_l}\!\!\!\!|\varphi|^2\,d\xi\,dt + \int_{0}^{\ell_0}\!\!|\varphi(\xi,0)|^2\,\,d\xi
        + \iint_{\mathcal{O}_l}\!\!\!\!|\phi|^2 \,d\xi\,dt + \int_{\ell_0}^{L}\!\!|\phi(\xi,0)|^2\,d\xi\\
	&\qquad \leq C_1\left[\iint_{\mathcal{O}_l} |\varphi - \mathbb{P}^{\ell}_l{\varphi}|^2\,d\xi\,dt 
           + \iint_{\mathcal{O}_r} |\phi - \mathbb{P}^{\ell}_r{\phi}|^2\,d\xi\,dt\right].
	\end{alignedat}
        \end{equation}
 	If \eqref{am1800} does not hold, there must exist $(\varphi_{T,1},\phi_{T,1}), (\varphi_{T,2},\phi_{T,2}), \dots$ in~$L^2(0,\ell_0)\times L^2(\ell_0,L)$ such that
        \begin{equation}\label{am105}
	\left\{
	\begin{alignedat}{2}
 	&\iint_{\mathcal{O}_l}\!\!\!\!|\varphi_n|^2\,d\xi\,dt + \int_{0}^{\ell_0}|\varphi_n(\xi,0)|^2\,\,d\xi + \iint_{\mathcal{O}_r}\!\!\!\!|\phi_n|^2 		\,d\xi\,dt + \int_{\ell_0}^{L}|\phi_n(\xi,0)|^2\,d\xi = 1 \ \text{ and} \\
	\noalign{\smallskip}\displaystyle	
	&\iint_{\mathcal{O}_l} |\varphi_n - \mathbb{P}^{\ell}_l{\varphi_n}|^2\,d\xi\,dt
        + \iint_{\mathcal{O}_r} |\phi _n- \mathbb{P}^{\ell}_r{\phi_n}|^2\,d\xi\,dt \leq \frac{1}{n}
	\end{alignedat}
	\right.
        \end{equation}
for all $n \geq 1$.
        Noting that
        \[
        \begin{array}{l}
        \dfrac{1}{2}\displaystyle\iint_{\mathcal{O}_{l}}|\mathbb{P}^{\ell}_{l}\varphi_{n}|^2\,d\xi\,dt +  \dfrac{1}
        {2}\iint_{\mathcal{O}_{r}}|\mathbb{P}^{\ell}_{r}\phi_{n}|^2\,d\xi\,dt \\
        \quad \leq\displaystyle\iint_{\mathcal{O}_{l}}\left[|\varphi_{n}|^2 +|\varphi_{n}
        - \mathbb{P}^{\ell}_{l}\varphi_{n}|^2\right]\,d\xi\,d
        + \displaystyle\iint_{\mathcal{O}_{r}}\left[|\phi_{n}|^2 + |\phi_{n} 
        - \mathbb{P}^{\ell}_{l}\phi_{n}|^2\right]\,d\xi\,dt ,
        \end{array}
        \]  
we easily get from~\eqref{am105} that the $(\beta^{\ell}_{l}(\varphi_{n}),\beta^{\ell}_{r}(\phi_{n}))$ are uniformly bounded in $\mathbb{R}^2$.
	Consequently, there exist a subsequence (again indexed by $n$) and a couple $(\beta_{l}^{*},\beta_{r}^{*})\in \mathbb{R}^2$ such that
        \begin{equation}\label{am106}
        (\beta^{\ell}_{l}(\varphi_{n}),\beta^{\ell}_{r}(\phi_{n}))\rightarrow (\beta_{l}^{*},\beta_{r}^{*})\quad\mbox{in}\,\,\,
        \mathbb{R}^2.
        \end{equation}
	It is clear from \eqref{am14}, \eqref{am15} and~\eqref{am105}$_1$ that, at least for a new subsequence, one has 
        \[
\begin{alignedat}{2}
       &\varphi_n\rightarrow \varphi ~~\mbox{weakly in}~~L^2(\sigma, T - \sigma; H^2(0,\ell_0)\cap H_0^1(0,\ell_0)),\\
       &\varphi_{n,t} \rightarrow \varphi_t ~~\mbox{weakly in}~~L^2(\sigma, T - \sigma; H^{-1}(0,\ell_0)),\\
       &\zeta_n\rightarrow \zeta ~~\mbox{weakly in}~~L^2(\sigma, T - \sigma; H^2(\ell_0,L)\cap H_0^1(\ell_0,L)),\\
       &\zeta_{n,t} \rightarrow \zeta_t ~~\mbox{weakly in}~~L^2(\sigma, T - \sigma; H^{-1}(\ell_0,L)),
\end{alignedat}
        \]  
for all $\sigma > 0$ small enough.
	Obviously, we have
        \begin{equation}\label{am109}
\left\{
\begin{array}{lll}
(M^{\ell}_{l})^{*}(\varphi) = 0 &\mbox{in} & Q_{0,l},\\
(M^{\ell}_{r})^{*}(\phi) = 0 &\mbox{in} & Q_{0,r},\\
\varphi(\cdot,0) = \varphi(\cdot,\ell_0) = 0 &\mbox{in} & (0,T),\\
\varphi(\cdot,\ell_0) = \varphi(\cdot,L) = 0 &\mbox{in} & (0,T). 
\end{array}
\right.
        \end{equation}
	Moreover, since $(\varphi_n,\phi_n) = (\varphi_n - \mathbb{P}^{\ell}_{l}\varphi_n,\phi_n - \mathbb{P}^{\ell}_{r}\phi_n) + (\mathbb{P}^{\ell}_{l}\varphi_n,\mathbb{P}^{\ell}_{r}\phi_n)$ in $\mathcal{O}_{l}\times \mathcal{O}_{r}$, using~\eqref{am105}$_2$ and~\eqref{am106}, it is also true that 
        \[
	(\varphi_n,\phi_n) \rightarrow (\mathbb{P}_{l}^{*}\varphi,\mathbb{P}_{r}^{*}\phi) \quad \mbox{strongly in}\,\,\,
	L^2(\mathcal{O}_{l})\times L^2(\mathcal{O}_{r}),
        \]
where $(\mathbb{P}_{l}^{*}\varphi,\mathbb{P}_{r}^{*}\phi) = (\beta_{l}^{*}\psi_{\ell},\beta_{r}^{*}\zeta_{\ell})$. 

	We have from \eqref{am9} and \eqref{am109} that $((M^{\ell}_{l})^{*}(\varphi - \mathbb{P}_{l}^{*}\varphi),(M^{\ell}_{r})^{*}(\phi - \mathbb{P}_{r}^{*}\phi)) = (0,0)$ in $Q_{0,l}\times Q_{0,r}$ and also, in view of \eqref{am105} and \eqref{am106}, $(\varphi - \mathbb{P}^{*}_{l}\varphi,\phi - \mathbb{P}^{*}_{l}\phi) = (0,0)$ in $\mathcal{O}_{l}\times \mathcal{O}_{r}$.
	Then, by applying a classical {\it unique continuation} argument, we conclude that $(\varphi,\phi) = (\mathbb{P}^{*}_{l}\varphi,\mathbb{P}^{*}_{r}\phi)$ in $Q_{0,l}\times Q_{0,r}$. However, this implies $(\varphi,\phi) = (0,0)$ in $Q_{0,l}\times Q_{0,r}$, since 
        \[
	(0,0) = (\varphi(\ell_0,\cdot),\phi(\ell_0,\cdot)) = (\beta_{l}^{*}\psi_{\ell}(\ell_0,\cdot),\beta_{r}^{*}\zeta_{\ell}
	(\ell_0,\cdot)) = (\beta^{*}_{l},\beta^{*}_{r}).
        \]
 	In other words,
        \[
	(\varphi_n,\phi_n) \rightarrow (0,0) \quad \mbox{in}\,\, L^2(\mathcal{O}_l)\times L^2(\mathcal{O}_r).
        \]
   
	Then, taking into account \eqref{am16} and \eqref{am105}, we see that
        \[
	\iint_{\mathcal{O}_l}|\varphi_{n}|^2\,d\xi\,dt + \int_{0}^{\ell_0}|\varphi_{n}(\xi,0)|^2\,\,d\xi
	+ \iint_{\mathcal{O}_r}|\phi_{n}|^2 \,d\xi\,dt + \int_{\ell_0}^{L}|\phi_{n}(\xi,0)|^2\,d\xi\rightarrow 0,
        \]
which is obviously absurd.
    
	This proves~\eqref{am1800}.
	The remaining terms in $I(\varphi_T,\phi_T)$ can also be bounded by the right hand side of~\eqref{am18}, as an immediate consequence of Proposition~\ref{proplimitation_coef}.
\end{proof}
 
       
\subsection{Approximate controllability problem with linear constraint}\label{approximatedsection}
       
        	In this section, we prove the approximate controllability of \eqref{am7} subject to the linear constraint~\eqref{am12}.
       	More precisely, the following holds:
	
\begin{prop}\label{controlconstr}
       	Assume that $R>0$, $\ell_0\in (\ell_l,\ell_r)$ and $\ell\in C^1([0,T])$ satisfies $\ell_l < \ell(t) < \ell_r$ for all $t\in [0,T]$,  $\ell(0) = \ell_0$ and $\|\ell'\|_{C^0([0,T])} \leq R$. 
	Then, for any $\varepsilon > 0$, any data $p_0\in H_0^1(0,\ell_0)$ and $q_0\in H_0^1(\ell_0,L)$ and  any $\ell_T\in (\ell_l,\ell_r)$, there exist controls $(h^{\ell}_{l,\varepsilon},h^{\ell}_{r,\varepsilon})\in L^2(\mathcal{O}_l)\times L^2(\mathcal{O}_r)$ and associated solutions to \eqref{am7}, with 	
        \[
   	\left\{
   		\begin{array}{l}
			p\in L^2(0,T;H^2(0,\ell_0))\cap H^1(0,T;L^2(0,\ell_0)),\\
			\noalign{\smallskip}\displaystyle	
        			q\in L^2(0,T;H^2(\ell_0,L))\cap H^1(0,T;L^2(\ell_0,L)),
   		\end{array}
  	\right.
        \]
satisfying the approximate controllability condition  
      	\begin{equation}\label{approximatedcontrol}
      	\|(p(\cdot,T),q(\cdot,T)\|_{L^2(0,\ell_0)\times L^2(\ell_0,L)} \leq \varepsilon
      	\end{equation}
and the linear constraint \eqref{am12}.
	Furthermore, the controls can be chosen satisfying
        \begin{equation}\label{am101}
	\begin{alignedat}{2}
	\|(h^{\ell}_{l,\varepsilon}1_{\mathcal{\omega}_l}),(h^{\ell}_{r,\varepsilon}1_{\mathcal{\omega}_l}))\|_{L^2(Q_{0,l})\times 
	L^2(Q_{0,r})}
	\leq C\left(\|(p_0,q_0)\|_{L^2 \times L^2} + |\ell_0 -\ell_T|\right),
	\end{alignedat}
        \end{equation}
where the constant $C > 0$ depends only on $\ell_l,\ell_r, \omega_l,\omega_r, T$ and $R$.             
\end{prop}

\begin{proof}
    	Let us first introduce the notation
        \[
    	M_\ell := \ell_T - \ell_0 - \int_{0}^{\ell_0}p_0(\xi)\psi_{\ell}(\xi,0)\,d\xi - \int_{\ell_0}^L q_0(\xi)\zeta_{\ell}(\xi,0)\,d\xi,
        \]
where the couple $(\psi_{\ell},\zeta_{\ell})$ is the unique solution to \eqref{am9}.
    	
	Now, for any given $\varepsilon > 0$, let us introduce the functional $J_{\ell,\varepsilon}: L^2(0,\ell_0)\times L^2(\ell_0,L) \mapsto\mathbb{R}$, defined as follows: 
	given $(\varphi_T,\phi_T)\in  L^2(0,\ell_0)\times L^2(\ell_0,L)$, we have
	\begin{equation}\label{am100}
	\begin{alignedat}{2}
    		J_{\ell,\varepsilon}(\varphi_{T},\phi_{T}) := &  \iint_{\mathcal{O}_l}|\varphi - \mathbb{P}^{\ell}_l\varphi|^2\,d\xi\,dt 
		+ \iint_{\mathcal{O}_r}|\phi - \mathbb{P}^{\ell}_r\phi|^2\,d\xi\,dt 
		+ \dfrac{\varepsilon}{2}\|(\varphi_T,\phi_T)\|_{L^2\times L^2}\\
    		& - \int_{0}^{\ell_0}p_0(\xi)\varphi(\xi,0)\,d\xi - \int_{\ell_0}^L q_0(\xi)\phi(\xi,0)\,d\xi   
		- \left[\beta^{\ell}_l(\varphi) + \beta^{\ell}_r(\phi)\right]\frac{M_\ell}{2},
	\end{alignedat}
        \end{equation}
where the couple $(\varphi,\phi)$ satisfies \eqref{am999}--\eqref{am1000}.
	
    	Using H\"older and Young inequalities, it is not difficult to check that $J_{\ell,\varepsilon}$ is a continuous, coercive and strictly convex functional.
	Therefore, $J_{\ell,\varepsilon}$ possesses a unique minimizer $(\varphi_{T}^{\varepsilon},\phi_{T}^{\varepsilon})\in L^2(0,\ell_0)\times L^2(\ell_0,L)$. The  corresponding solution to \eqref{am1000} will be denoted by $(\varphi_{\varepsilon},\phi_{\varepsilon})$.
	Then
	\begin{align}\label{eulerlagrange}	
	J_{\ell,\varepsilon}'(\varphi_{T}^{\varepsilon},\phi_{T}^{\varepsilon})(\varphi_{T},\phi_{T}) = 0 \quad \forall (\varphi_{T},
	\phi_{T})\in L^2(0,\ell_0)\times L^2(\ell_0,L),
	\end{align}	
where
        \[
	\begin{alignedat}{2}
	 J_{\ell,\varepsilon}'(\varphi_{T,\varepsilon},\phi_{T,\varepsilon})&(\varphi_{T},\phi_{T})
	 =  \iint_{\mathcal{O}_l}[\varphi_{\varepsilon} - \mathbb{P}^{\ell}_l(\varphi_{\varepsilon})]\varphi\,d\xi\,dt     
	     + \iint_{\mathcal{O}_r}[\phi_{\varepsilon} - \mathbb{P}^{\ell}_r(\phi_{\varepsilon})]\phi\,d\xi\,dt \nonumber\\
	     & + \dfrac{\varepsilon}{2\|\varphi_{T}^{\varepsilon}\|_{L^2}}\int_{0}^{\ell_0}\varphi_{T}^{\varepsilon}(\xi)\varphi_{T}
	     (\xi)\,d\xi 
	     + \dfrac{\varepsilon}{2\|\phi_{T}^{\varepsilon}\|_{L^2}}\int_{\ell_0}^{L}\phi_{T}^{\varepsilon}(\xi)\phi_{T}(\xi)\,d\xi\\
   	&- \int_{0}^{\ell_0}p_0(\xi)\varphi(\xi,0)\,d\xi - \int_{\ell_0}^Lq_0(\xi)\phi(\xi,0)\,d\xi - 
	\left[\beta^{\ell}_l(\varphi) + \beta^{\ell}_r(\phi)\right]\frac{M_\ell}{2},\nonumber
	\end{alignedat}
        \]
	Here, we have used the fact that $\langle\varphi_{\varepsilon} - \mathbb{P}^{\ell}_l(\varphi_{\varepsilon}),\mathbb{P}^{\ell}_l(\varphi)\rangle_{L^2(\mathcal{O}_l)}=0$ and $\langle\phi_{\varepsilon} - \mathbb{P}^{\ell}_r(\phi_{\varepsilon}),\mathbb{P}^{\ell}_r(\phi)\rangle_{L^2(\mathcal{O}_r)}=0$.
	
        Let us introduce
        \begin{equation}\label{am22}
	h^{\ell}_{l,\varepsilon}:= \left[\mathbb{P}^{\ell}_l({\varphi_{\varepsilon}}) - \varphi_{\varepsilon}\right] + {M_\ell\over2}
	{\psi_{\ell}\over \|\psi_{\ell}\|^2_{L^2(\mathcal{O}_l)}} \ \text{ and } \ 
	h^{\ell}_{r,\varepsilon} := \left[\mathbb{P}^{\ell}_r({\phi_{\varepsilon}}) - \phi_{\varepsilon}\right] + {M_\ell\over2}
	{\zeta_{\ell}\over \|\zeta_{\ell}\|^2_{L^2(\mathcal{O}_r)}}.
        \end{equation}

     	Let $(\varphi_{T},\phi_{T})\in L^2(0,\ell_0)\times  L^2(\ell_0,L)$ be given and let $(p,q)$ be  the solution to \eqref{am7} associated to the control pair $(h^{\ell}_{l,\varepsilon},h^{\ell}_{r,\varepsilon})$.
	Then, multiplying $\eqref{am7}$ by the solution $(\varphi,\phi)$ to $\eqref{am1000}$ and integrating in~$Q_{0,l}$ and~$Q_{0,r}$, we obtain
        \begin{equation}\label{am21}
\begin{alignedat}{2}
   	\iint_{\mathcal{O}_l}h^{\ell}_{l,\varepsilon}\varphi\,d\xi\,dt + \iint_{\mathcal{O}_r}h^{\ell}_{r,\varepsilon}\phi\,d\xi\,dt =&~ \int_{0}^{\ell_0}\left[p(\xi,T)\varphi(\xi,T) - p_0(\xi)\varphi(\xi,0)\right]d\xi \\
	&+ \int_{\ell_0}^{L}\left[q(\xi,T)\phi(\xi,T) - q_0(\xi)\phi(\xi,0)\right]d\xi.\\
  	 \end{alignedat}    
        \end{equation}
   	Taking into account \eqref{am22} and comparing \eqref{eulerlagrange} with \eqref{am21}, we get
        \[
	\begin{alignedat}{2}
	     \int_{0}^{\ell_0}p(\xi,T)\varphi_{T}(\xi)\,d\xi + \int_{\ell_0}^{L}q(\xi,T)\phi_{T}(\xi)\,d\xi = &
	     {\varepsilon\over2}\left(\int_{0}^{\ell_0}{\varphi_{T}^{\varepsilon}(\xi)\over \|\varphi_{T}^{\varepsilon}\|_{L^2}}	
	     \varphi_{T}(\xi)\,d\xi +
	     \int_{\ell_0}^{L}{\phi_{T}^{\varepsilon}(\xi)\over\|\phi_{T}^{\varepsilon}\|_{L^2}}\phi_{T}(\xi)\,d\xi \right),
	\end{alignedat}
        \]
for all $(\varphi_{T},\phi_{T})\in L^2(0,\ell_0)\times L^2(\ell_0,L)$.
	Therefore, the approximate controllability condition \eqref{approximatedcontrol} follows.
	Since we also have  
        \[
  	\begin{alignedat}{2}
    	\iint_{\mathcal{O}_l}h^{\ell}_{l,\varepsilon}\psi_{\ell}\,d\xi\,dt 
	+ \iint_{\mathcal{O}_r}h^{\ell}_{r,\varepsilon}\zeta_{\ell}\,d\xi\,dt 
	=&~ \iint_{\mathcal{O}_l}\left[\mathbb{P}^{\ell}_l(\varphi_{\varepsilon}) 
	- \varphi_{\varepsilon}\right]\psi_{\ell}\,d\xi\,dt + \frac{M_\ell}{2} \\
    	&+ \iint_{\mathcal{O}_r}\left[\mathbb{P}^{\ell}_r(\phi_{\varepsilon}) 
	- \phi_{\varepsilon}\right]\zeta_{\ell}\,d\xi\,dt + \frac{M_\ell}{2}\\
        =&~ M_\ell,
     	\end{alignedat}
        \]  
  	the pair $(h^{\ell}_{l,\varepsilon},h^{\ell}_{r,\varepsilon})$ satisfies \eqref{am12} and, consequently, $\mathcal{L}(T) = \ell_T$.
   
   	Finally, due to the fact that $(\varphi_{T,\varepsilon},\phi_{T,\varepsilon})$ is the minimum of $J_{\ell,\varepsilon}$, we have the inequality $J_{\ell,\varepsilon}(\varphi_{T}^{\varepsilon},\phi_{T}^{\varepsilon}) \leq J_{\ell,\varepsilon}(0,0) = 0$.
	Using this fact and the definition of $M_\ell$ and~\eqref{am18}, we deduce that there exist positive constants $C$ (depending on $\ell_l, \ell_r,R, \omega_l,\omega_r$ and $T$) such that
        \[	
     	\begin{alignedat}{2}
        \|(\varphi_{\varepsilon} - \mathbb{P}^{\ell}_l(\varphi_{\varepsilon}))\|_{L^2(\mathcal{O}_l)} 
        + \|(\phi_{\varepsilon} - \mathbb{P}^{\ell}_r(\phi_{\varepsilon}))\|_{L^2(\mathcal{O}_r)}
        \leq ~ C\left(\|p_0\|_{L^2(0,\ell_0)} + \|q_0\|_{L^2(\ell_0,L)} + |\ell_0 - \ell_T|\right) 
     	\end{alignedat}
        \]  
and
        \[
     	\left.
     	\begin{alignedat}{2}
      	\|h^{\ell}_{l,\varepsilon}\|_{L^2(\mathcal{O}_l)} +  \|h^{\ell}_{l,\varepsilon}\|_{L^2(\mathcal{O}_r)}
	&\leq C \left(\|(\varphi_{\varepsilon} - \mathbb{P}^{\ell}_l(\varphi_{\varepsilon}))\|_{L^2(\mathcal{O}_l)} 
        + \|(\phi_{\varepsilon} - \mathbb{P}^{\ell}_r(\phi_{\varepsilon}))\|_{L^2(\mathcal{O}_r)} + |M_\ell| \right) \\
      	&\leq   C\left(\|p_0\|_{L^2(0,\ell_0)} + \|q_0\|_{L^2(\ell_0,L)} + |\ell_0 -\ell_T|\right).
     	\end{alignedat}
     	\right.
        \]
        
	This ends the proof.
\end{proof}


\section{Controllability of the two-phase Stefan problem}\label{sec:nonlinear}

	In this section we prove Theorem~\ref{twophase1}.
	The proof relies on a fixed-point argument. 
	First, it will be convenient to recall some regularity properties for linear parabolic systems.
	
	Due to Propositions~\ref{appendix_thm_free} and~\ref{appendix_thm_free2}, the smoothing effect of~\eqref{am4} implies that we can assume that 
	$(p_0,q_0)\in  W_0^{1,4}(0,\ell_0) \times W_0^{1,4}(\ell_0,L)$ and we can consider smallness assumptions of $(p_0,q_0)$ in this space.

  
\subsection{A regularity property}\label{regularityy}

	First of all, let us assume that $(p_0 , q_0) \in W_0^{1,4}(0,\ell_0)\times  W_0^{1,4}(\ell_0,L)$. For any open interval $ I \subset\mathbb{R}$, let us introduce the Banach space
        \[  
        X^{4}(0,T;I) : = L^{4}(0,T;W^{2,4}(I))\cap W^{1,4}(0,T; L^4(I)).
        \]       
        On the other hand, let us consider the cylinder $G_l:= (\ell_l,\ell_0)\times (0,T)$, the H\"older semi-norms
        \[
        \langle u \rangle^{\kappa}_{\xi,G_l} := \sup_{{(\xi,t),(\xi',t)\in \overline{G}_{l}}\atop{\xi\neq \xi'}} \dfrac{|u(\xi,t) - u(\xi',t)|}
        {|\xi - \xi'|^{\kappa}}
        \]
and
        \[
        \langle u \rangle^{\kappa}_{t,G_l} := \sup_{{(\xi,t),(\xi,t')\in \overline{G}_{l}}\atop{t\neq t'}} 
        \dfrac{|u(\xi,t) - u(\xi,t')|}{|t - t'|^{\kappa}}
        \]
where $0 < \kappa < 1$ and the space $C^{\kappa,\kappa/2}(\overline{G}_{l})$ formed by the functions $u \in C^0(\overline{G}_{l})$ whose corresponding $\langle u \rangle^{\kappa}_{\xi,G_{l}}$ and $\langle u \rangle^{\kappa/2}_{t,G_{l}}$ are finite. It is known that $C^{\kappa,\kappa/2}(\overline{G}_{l})$ is a Banach space (see~\cite{ladyzhenskaya}) with the following norm: 
        \[
        \|u\|_{\kappa,\kappa/2;\overline{G}_l} := \|u\|_{C^0(\overline{G}_l)} 
        + \langle u \rangle^{\kappa}_{\xi,G_l}+ \langle u \rangle^{\kappa/2}_{t,G_l}.
        \]
	Finally, let us introduce the Banach space 
        \[
        C^{1 + \kappa,(1 + \kappa)/2}(\overline{G}_{l}) := \{ u \in  C^0(\overline{G}_{l}) 
        : u_{\xi} \in C^{\kappa,\kappa/2}(\overline{G}_{l}),\ \ \langle u \rangle^{(1 + \kappa)/2}_{t,G_l} < +\infty \}.
        \]

	Obviously, we can introduce similar quantities and spaces for functions defined in~$G_{r} := (\ell_0,\ell_r)\times (0,T)$.
	The following result holds:
	
\begin{lem}\label{lemmaregul}
 	Let us assume that $\ell_0, \ell_T \in (\ell_l,\ell_r)$ and $(p_0 , q_0) \in W_0^{1,4}(0,\ell_0)\times W_0^{1,4}(\ell_0,L)$. 
	Then, the states $(p,q)$, furnished by Proposition~\ref{controlconstr} satisfy
	\[
         (p,q) \in C^{1 + \kappa,(1 + \kappa)/2}(\overline{G}_{l})\times 
         C^{1 + \kappa,(1 + \kappa)/2}(\overline{G}_{r})~~\mbox{for}~\kappa = 1/4.
        \]
        Furthermore, there exists $C > 0$, depending on $\ell_l$, $\ell_r$, $\omega_l$, $\omega_r$, $T$ and~$R$, such that
        \begin{equation}\label{regularityyyy}
	\|p\|_{1+\kappa,(1 + \kappa)/2;\overline G_{l}} 
	+ \|q\|_{1+\kappa,(1 + \kappa)/2;\overline G_{r}} 
	\leq C\left(\|(p_0,q_0)\|_{W_0^{1,4}\times W_0^{1,4}} + |\ell_0 - \ell_T|\right).
        \end{equation}
\end{lem}

\begin{proof}
      Clearly, due the regularity of~$p_0$, there exists a function $f \in X^{4}(0,T;(0,\ell_0))$ such that $f(0,t) = f(\ell_0,t) = 0$, for $t \in (0,T)$, and $f(\xi,0) = p_0(\xi)$, for $\xi\in (0,\ell_0)$.
      Consequently, the state $p$, provided by Proposition~\ref{controlconstr}, can be written in the form $p = y + f$, where $y \in L^2(0,T; H^2(0,\ell_0))\cap H^1(0,T;L^2(0,\ell_0))$ is the unique strong solution of the following problem:
        \begin{equation}\label{solutionnnu}
            \left\{
            \begin{array}{lll}
            y_t - d^{\ell}_ly_{\xi\xi} + b^{\ell}_ly_{\xi} = F &\mbox{in} & Q_{0,l}\\
            y(0,\cdot) = y(\ell_0,\cdot) = 0 &\mbox{in} & (0,T),\\
            y(\cdot,0) = 0 &\mbox{in} & (0,\ell_0), 
            \end{array}
            \right.
        \end{equation}
where $ F = h^{\ell}_{l,\varepsilon}1_{\omega_l} - f_t + d^{\ell}_lf_{\xi\xi} - b^{\ell}_lf_{\xi}$. 
		
	Now, let $\sigma > 0$ be such that $\omega_l\subset\subset (0,\ell_l - \sigma)$ and, moreover, $G^{\sigma}_{l} := (\ell_l - \sigma,\ell_l + \sigma)\times (0,T)\subset Q_{0,l}$.
	We can easily check that $F \in L^4(0,T; L^4(\ell_l - \sigma,\ell_l + \sigma))$.
	Therefore, from local parabolic regularity results, we obtain that $y \in X^{4}(0,T;(\ell_l - \sigma/2, \ell_l + \sigma/2))$ and
        \[
        \|y\|_{X^{4}(0,T;(\ell_l - \sigma/2, \ell_l + \sigma/2))}\leq
        C\left(\|F\|_{L^4(0,T; L^4(\ell_l - \sigma,\ell_l + \sigma))} 
        + \|y\|_{L^2(0,T; H^2(0,\ell_0))\cap H^1(0,T;L^2(0,\ell_0))} \right),
        \]		
	where $C$ only depends on $\|d^{\ell}_l\|_{\infty}$, $\|b^{\ell}_l\|_{\infty}$, $\ell_l$, $\ell_0$ and~$\sigma$.
	
        Next, using standard parabolic energy estimates and~\eqref{am101}, we get
        \[
        \|y\|_{{X^{4}(0,T;(\ell_l - \sigma/2, \ell_l + \sigma/2))}} \leq 
        C\left(\|(p_0,q_0)\|_{ W_0^{1,4}\times W_0^{1,4}} + |\ell_0 - \ell_T|\right)
        \]
for some $C > 0$ as above.
	Here, we have used that $\|d^{\ell}_l\|_{\infty}$ and~$\|b^{\ell}_l\|_{\infty}$  are bounded in terms of $R$.
	Finally, using this inequality, the regularity of the trace $y(\ell_l,\cdot)$, the fact that $y$ is a strong solution to \eqref{solutionnnu} and~\cite[Propositions $9.2.3$ and $9.2.5$]{wu}, we conclude that $y\in X^{4}(0,T;(\ell_l,\ell_0))$ and, moreover,
        \begin{equation}\label{am112}
        \|y\|_{X^{4}(0,T;(\ell_l,\ell_0))} \leq C\left(\|(p_0,q_0)\|_{W_0^{1,4}\times W_0^{1,4}} + |\ell_0 - \ell_T|\right)
        \end{equation}
 for a new $C > 0$.
 
	In a similar way, we can write $ q = z + g$, where $g\in X^{4}(0,T;(\ell_0,L))$ is a shift function for the initial data $q_0$ and $z\in X^{4}(0,T;(\ell_0,\ell_r))$ satisfies 
        \begin{equation}\label{am113}
        \|z\|_{X^{4}(0,T;(\ell_0,\ell_r))} \leq C\left(\|(p_0,q_0)\|_{W_0^{1,4}\times W_0^{1,4}} + |\ell_0 - \ell_T|\right).
        \end{equation}
		
	Then, the estimate in \eqref{regularityyyy} is a immediate consequence of \eqref{am112}-\eqref{am113} and the following embedding 
	from \cite[Lemma $2.2$]{bodart}
\[
       X^{4}(0,T;(\ell_l,\ell_0))\times X^{4}(0,T;(\ell_0,\ell_r)) \hookrightarrow   C^{1 + \kappa,(1 + \kappa)/2}(\overline{G}_{l})\times C^{1 + \kappa,(1 + \kappa)/2}(\overline{G}_{r}),
\] 
    where $\kappa = 1/4$.  
\end{proof}\\
     Let us introduce the function $\theta : [0,T] \mapsto\mathbb{R}$, given by 
\begin{equation}\label{thetadefi}
                    \theta(t) = d_rq_{\xi}(\ell_0,t) - d_lp_{\xi}(\ell_0,t).
\end{equation}
    Then, as an immediate consequence of \eqref{regularityyyy}, we get that $\theta\in C^{1/8}([0,T])$ and, moreover, there exists a positive constant $C$ (depending on $\ell_l,\ell_r,\omega_l,\omega_r, T$ and $R$) such that
\begin{equation}\label{am104}
        \|\theta\|_{C^{1/8}([0,T])} \leq C\left(\|(p_0,q_0)\|_{W_0^{1,4}\times  W_0^{1,4}}  + |\ell_0 - \ell_T|\right).
\end{equation}


\subsection{A fixed-point argument} 

	In this Section we will achieve the proof of Theorem \ref{twophase1}. It will be a consequence of the following uniform approximate controllability result:	 
	\begin{thm}\label{twophase1_approx}
    Assume that $R>0$ is given. 
	Then, there exists $\delta > 0$ such that, for any $p_0\in W_0^{1,4}(0,\ell_0)$ with $p_0 \geq 0$, any $q_0\in W_0^{1,4}(\ell_0,L)$ with $q_0 \leq 0$, any $\ell_0, \ell_T\in (\ell_l,\ell_r)$ satisfying 
$$
	\|p_0\|_{W_0^{1,4}(0,\ell_0)} + \|q_0\|_{W_0^{1,4}(\ell_0,L)} + |\ell_0 -\ell_T| \leq \delta
$$	
	and any $\varepsilon > 0$, there exist controls $(h_l^\varepsilon,h_r^\varepsilon)\in L^2(\mathcal{O}_l)\times L^2(\mathcal{O}_r)$ and associated solutions to~\eqref{am4}, with 	
\begin{equation}\label{regularityfixedpoint}
   	\left\{
   		\begin{array}{l}
     			\ell_{\varepsilon}\in C^{1}([0,T])~\mbox{and}~ \ell_{\varepsilon}(t) \in (\ell_l,\ell_r)\,\forall\, t\in [0,T],~ \|\ell'_{\varepsilon}\|_{C^0([0,T])}\leq R,\\
     			\noalign{\smallskip}\displaystyle	
			p_{\varepsilon}\in L^2(0,T;H^2(0,\ell_0))\cap H^1(0,T;L^2(0,\ell_0)),\\
			\noalign{\smallskip}\displaystyle	
        			q_{\varepsilon}\in L^2(0,T;H^2(\ell_0,L))\cap H^1(0,T;L^2(\ell_0,L)),
   		\end{array}
  	 \right.
\end{equation}
      	 satisfying the exact-approximate controllability condition  
      	\begin{equation}\label{approximatedcontrol2}
     \ell_{\varepsilon}(T) = \ell_T ~~\mbox{and}~~\|(p_{\varepsilon}(\cdot,T),q_{\varepsilon}(\cdot,T))\|_{L^2(0,\ell_0)\times L^2(\ell_0,L)} \leq \varepsilon.
      	\end{equation}
     Moreover, the controls can be found satisfying the following uniform estimate with respect to $\varepsilon$:
\begin{equation}\label{am1012}
	\begin{alignedat}{2}
		\|(h_l^\varepsilon1_{\mathcal{\omega}_l}),(h_r^\varepsilon1_{\mathcal{\omega}_r}))\|_{L^2(Q_{0,l})\times L^2(Q_{0,r})}
		\leq C\left(\|(p_0,q_0)\|_{W^{1,4} \times W^{1,4}} + |\ell_0 -\ell_T|\right)
	\end{alignedat}
\end{equation}              
 for some positive $C$ (depending on $\ell_l,\ell_r,\omega_l,\omega_r, T$ and $R$).
\end{thm}
\begin{proof}
	Given $\ell_l < \tilde{\ell}_l < \tilde{\ell}_r < \ell_r$ and $R > 0$, we define the set:  
     $$
     \mathcal{A}_R := \{\ell\in C^{1}([0,T]): \tilde{\ell}_l \leq \ell(t) \leq \tilde{\ell}_r,~\forall t\in [0,T],~~\ell(0) = \ell_0,~ \|\ell'\|_{C^0([0,T])} \leq R\}.
     $$
     Obviously, $\mathcal{A}_R$ is a non-empty, closed and convex subset of $C^{1 }([0,T])$. Let us also introduce the mapping $\Lambda_{\varepsilon} : \mathcal{A}_R \mapsto C^{1}([0,T])$, given by
\begin{equation*}\label{fixedpointmap}
	      \Lambda_{\varepsilon}(\ell)= \mathcal{L},~~\hbox{with}~~\mathcal{L}(t) := \ell_0 - \displaystyle\int_{0}^{t}\left[d_l p_{\xi}(\ell_0, \tau ) - d_r q_{\xi}(\ell_0, \tau )\right]\,d\tau,
\end{equation*}
       where $(p,q)$ is the state associated to the control pair $(h^{\ell}_{l,\varepsilon},h^{\ell}_{r,\varepsilon})$ constructed as in the proof of Proposition \ref{controlconstr} (recall Lemma \ref{lemmaregul}) and, therefore, $\mathcal{L}(T) =\ell_T$. Thanks to  \eqref{thetadefi} and \eqref{am104}, we have that $\mathcal{L}\in C^{1 }([0,T])$.

       Let us check that $\Lambda_{\varepsilon}$ satisfies the conditions of Schauder's Fixed-Point Theorem.
     \begin{itemize}
     
    \item {\it $\Lambda_{\varepsilon}$ is continuous}: Indeed, let the $\ell_n (n \geq 1)$ and $\ell$ belong to $\mathcal{A}_{R}$ and assume that $\ell_n \rightarrow \ell$ in $C^{1}([0,T])$. We must prove that $\Lambda_{\varepsilon}(\ell_n) \rightarrow \Lambda_{\varepsilon}(\ell)$ in $C^{1}([0,T])$. To that end, we will first prove that the corresponding solutions to \eqref{am9} satisfy
    \begin{equation}\label{conv:psi-zeta}
    (\psi_{\ell_n},\zeta_{\ell_n}) \rightarrow (\psi_{\ell},\zeta_{\ell})\quad\mbox{strongly in}\quad L^2(Q_{0,l})\times L^2(Q_{0,r}).
    \end{equation}
 	Let $f \in L^2(0,T; H^2(0,\ell_0))\times H^1(0,T;L^2(0,\ell_0))$ be such that $f(0,\cdot) = 0$ and   $f(\ell_0,\cdot) = 1$  on $(0,T)$ and let us put $\psi_{\ell_n} = \Psi_{\ell_n} + f$ and $\psi_{\ell} = \Psi_{\ell} + f$. It is then clear that $y_{\ell_n} :=  \Psi_{\ell_n} - \Psi_{\ell}$ is the unique weak solution to:
\begin{equation*}
 	 \left\{
 	 \begin{array}{lll}
 	 (M^{\ell_n}_{l})^{*}(y_{\ell_n}) = F_{\ell_n} & \mbox{in} & Q_{0,l},\\
     y_{\ell_n}(0,\cdot) = y_{\ell_n}(\ell_0,\cdot) = 0 & \mbox{on} & (0,T),\\
 	 y_{\ell_n}(\cdot,0) = 0 & \mbox{in} & (0,\ell_0),
 	 \end{array}
 	 \right.
\end{equation*}
		where $F_{\ell_n} \in L^2(0,T ;H^{-1}(0,\ell_0))$ is given by
\[
   \begin{alignedat}{2}
   F_{\ell_n} := &  (d^{\ell}_{l,\xi\xi} - d^{\ell_n}_{l,\xi\xi})f + 2(d^{\ell}_{l,\xi} - d^{\ell_{n}}_{l,\xi})f_{\xi} + (d^{\ell}_{l} - d^{\ell_{n}}_{l})f_{\xi\xi} + (b^{\ell}_{l,\xi} - b^{\ell_n}_{l,\xi})f \\
    & + (b^{\ell}_{l} - b^{\ell_n}_{l})f_{\xi} + ((d^{\ell}_{l} - d^{\ell_n}_{l})\Psi_{\ell})_{\xi\xi} + ((b^{\ell}_{l} - b^{\ell_n}_{l})\Psi_{\ell})_{\xi}.   
   \end{alignedat}
\]   
	 Then, using the fact that $(d^{\ell_n}_{l},b^{\ell_n}_{l})$ are uniformly bounded in the space 
	$C^{1}(\overline{Q}_{0,l})\times C^{0}(\overline{Q}_{0,l})$ and the $(d^{\ell_n}_{l,\xi\xi},b^{\ell_n}_{l,\xi})$ are uniformly bounded in  $C^{0}(\overline{Q}_{0,l})\times C^{0}(\overline{Q}_{0,l})$, the standard parabolic energy estimates and the regularity of the function $G$ (as well as the regularity of its inverse $G^{-1}$),  we get that $y_{\ell_n} \rightarrow 0$ strongly in  $L^{2}(Q_{0,l})$ which, in turn, implies 
$$		
	 \psi_{\ell_n}	\rightarrow \psi_{\ell}~~\mbox{strongly in}~~L^{2}(Q_{0,l}).
$$	
	Analogously, we can prove that $\zeta_{\ell_n} \rightarrow \zeta_{\ell}$ strongly in $L^{2}(Q_{0,r})$.
	
	Now, we recall that, for each $\varepsilon > 0$, there exists a unique $(\varphi_{T,\varepsilon}^{n},\phi^{n}_{T,\varepsilon})$ in $L^2(0,\ell_0)\times L^2(\ell_0,L)$ that minimizes the functional $J_{\ell_n,\varepsilon}$, defined in \eqref{am100}. Due to the facts that $J_{\ell_n,\varepsilon}(\varphi^{n}_{T,\varepsilon},\phi^{n}_{T,\varepsilon}) \leq 0$ and the constant appearing in the right side of \eqref{am18} does not depend on $n$, we get that the minimizers are uniformly bounded with respect to $n$ in the space $L^2(0,\ell_0)\times L^2(\ell_0,L)$ and the corresponding $(\varphi^{n}_{\varepsilon},\phi^{n}_{\varepsilon})$, solutions to \eqref{am1000}, are uniformly bounded  
	spaces given in \eqref{am999}. Therefore, there exist $(\varphi_{T,\varepsilon},\phi_{T,\varepsilon})$ in $L^2(0,\ell_0)\times L^2(\ell_0,L)$ and $(\varphi_{\varepsilon},\phi_{\varepsilon})$ in $L^2(Q_{0,l})\times L^2(Q_{0,r})$ such that, at least for a subsequence, one has
\begin{equation}\label{convergence}
             \left\{
             \begin{array}{l}
             (\varphi^{n}_{T,\varepsilon},\phi^{n}_{T,\varepsilon}) \rightarrow (\varphi_{T,\varepsilon},\phi_{T,\varepsilon})~~\mbox{weakly in}~~L^2(0,\ell_0)\times L^2(\ell_0,L),\\
             \\
              (\varphi_{\varepsilon}^{n}(\cdot,0),\phi_{\varepsilon}^{n}(\cdot,0)) \rightarrow (\varphi_{\varepsilon}(\cdot,0),\phi_{\varepsilon}(\cdot,0))~~\mbox{weakly in}~~L^2(0,\ell_0)\times L^2(\ell_0,L)~~\mbox{and}\\
              \\
             (\varphi^{n}_{\varepsilon},\phi^{n}_{\varepsilon}) \rightarrow (\varphi_{\varepsilon},\phi_{\varepsilon})~~\mbox{strongly in}~~L^2(Q_{0,l})\times L^2(Q_{0,r}).
             \end{array}
             \right.
\end{equation}	 
	
	We will show now that $(\varphi_{T,\varepsilon},\phi_{T,\varepsilon})$ is the unique minimizer of the functional $J_{\ell,\varepsilon}$. Indeed, we first note from the convergences in \eqref{conv:psi-zeta} and $\eqref{convergence}_3$ that  
	$$
	(\mathbb{P}^{\ell_n}_{l}(\varphi^{n}_{\varepsilon}),\mathbb{P}^{\ell_n}_{r}(\phi^{n}_{\varepsilon})) \rightarrow (\mathbb{P}^{\ell}_{l}(\varphi_{\varepsilon}),\mathbb{P}^{\ell}_{r}(\phi_{\varepsilon}))\quad  \hbox{strongly in}\quad L^2(Q_{0,l})\times L^2(Q_{0,r}).
	$$ 
	Then, using this fact and the weak convergences $\eqref{convergence}_{1,2}$,  we easily  get
\begin{equation}\label{minimizerfunc}
        J_{\ell,\varepsilon}(\varphi_{T,\varepsilon},\phi_{T,\varepsilon}) \leq \liminf_{n} J_{\ell_{n},\varepsilon}(\varphi^{n}_{T,\varepsilon},\phi^{n}_{T,\varepsilon}).
\end{equation}
	
	Now, let $(\varphi_T,\phi_T)$ be given in $L^2(0,\ell_0)\times L^2(\ell_0,L)$  and let the $(\varphi^{n},\phi^{n})$ be the solutions to the system \eqref{am1000}, with $\ell$ replaced by $\ell_n$, for $n = 1, 2, \ldots$. Then, using the same ideas that led to prove of \eqref{conv:psi-zeta}, we can ensure that the $(\varphi^{n},\phi^{n})$ converge strongly in $L^2(Q_{0,l})\times L^2(Q_{0,r})$, to the solution $(\varphi,\phi)$ to \eqref{am1000} and the $(\varphi^{n}(\cdot,0),\phi^{n}(\cdot,0))$ converge weakly in $L^2(0,\ell_0)\times L^2(\ell_0,L)$ to $(\varphi(\cdot,0),\phi(\cdot,0))$. Therefore, from \eqref{conv:psi-zeta} and \eqref{minimizerfunc} we deduce that
\begin{equation}\label{minimialfree}
            J_{\ell,\varepsilon}(\varphi_{T,\varepsilon},\phi_{T,\varepsilon}) \leq \liminf_{n} J_{\ell_{n},\varepsilon}(\varphi^{n}_{T,\varepsilon},\phi^{n}_{T,\varepsilon}) \leq \liminf_{n} J_{\ell_{n},\varepsilon}(\varphi_{T},\phi_{T}) = J_{\ell,\varepsilon}(\varphi_{T},\phi_{T}).
\end{equation}
    Since $(\varphi_T,\phi_T)\in L^2(0,\ell_0)\times L^2(\ell_0,L)$ is arbitrary, we conclude that $(\varphi_{T,\varepsilon},\phi_{T,\varepsilon})$ minimizes $J_{\ell,\varepsilon}$.
    
   Now, let us consider, for each $n$, the pair $(h^{\ell_n}_{l,\varepsilon}1_{\mathcal{\omega}_{l}},h^{\ell_n}_{r,\varepsilon}1_{\mathcal{\omega}_{r}})$ associated by Proposition \ref{controlconstr} to $\ell_n$. It follows easily from \eqref{conv:psi-zeta}, $\eqref{convergence}_3$ and \eqref{minimialfree} that 
\begin{equation}\label{convergencecontrol}
               (h^{\ell_n}_{l,\varepsilon}1_{\mathcal{\omega}_{l}},h^{\ell_n}_{r,\varepsilon}1_{\mathcal{\omega}_{r}}) \rightarrow (h^{\ell}_{l,\varepsilon}1_{\mathcal{\omega}_{l}},h^{\ell}_{r,\varepsilon}1_{\mathcal{\omega}_{r}})~~\mbox{strongly in}~~L^2(\mathcal{O}_{l})\times L^2(\mathcal{O}_{r}),
\end{equation}     
    where $(h^{\ell}_{l,\varepsilon}1_{\mathcal{\omega}_{l}},h^{\ell}_{r,\varepsilon}1_{\mathcal{\omega}_{r}})$ is the control corresponding to $\ell$. Let us denote by $(p^{n}_{\varepsilon},q^{n}_{\varepsilon})$ and $(p_{\varepsilon},q_{\varepsilon})$ the solutions to \eqref{am7} associated, respectively, to $(h^{\ell_n}_{l,\varepsilon}1_{\mathcal{\omega}_{l}},h^{\ell_n}_{r,\varepsilon}1_{\mathcal{\omega}_{r}})$ and $(h^{\ell}_{l,\varepsilon}1_{\mathcal{\omega}_{l}},h^{\ell}_{r,\varepsilon}1_{\mathcal{\omega}_{r}})$. Then, if we set $(y^{n},z^{n}):= (p_{\varepsilon}^{n} - p_{\varepsilon},q_{\varepsilon}^{n} - q_{\varepsilon})$ and $(w_{l}^{n}1_{\mathcal{\omega}_l},w_{r}^{n}1_{\mathcal{\omega}_r}) := (h^{\ell_n}_{l,\varepsilon}1_{\mathcal{\omega}_l} - h^{\ell}_{l,\varepsilon}1_{\mathcal{\omega}_l}, h^{\ell_n}_{r,\varepsilon}1_{\mathcal{\omega}_r} - h^{\ell}_{r,\varepsilon}1_{\mathcal{\omega}_r}) $, we find that:
\begin{equation}\label{difference}
              \left\{
              \begin{array}{lll}
              y^{n}_t - d^{\ell_n}_ly^{n}_{\xi\xi} + b_l^{\ell_n}y^{n}_{\xi} = w_{l}^{n}1_{\mathcal{\omega}_l} + F^{n}_{l} & \mbox{in} & Q_{0,l}, \\
              z^{n}_t - d^{\ell_n}_rz^{n}_{\xi\xi} + b_r^{\ell_n}z^{n}_{\xi} = w_{r}^{n}1_{\mathcal{\omega}_r} + F^{n}_{r}& \mbox{in} & Q_{0,r},\\
              y^{n}(0,\cdot) = y^{n}(\ell_0,\cdot) = z^{n}(\ell_0,\cdot) = z^{n}(L,\cdot) = 0 & \mbox{in} &  (0,T),\\ 
              y^{n}(\cdot,0) = 0 & \mbox{in} &  (0,\ell_0),\\
              z^{n}(\cdot,0) = 0 & \mbox{in} &  (\ell_0,L),
              \end{array}
              \right.
\end{equation}
       where  
       \[
       F^{n}_{l} :=  (d^{\ell_n}_l - d^{\ell}_l)p_{\varepsilon, \xi\xi} {\color{red}-} (b^{\ell_n}_l - b_l^{\ell})p_{\varepsilon, \xi}~~\mbox{and}~~F^{n}_{l} :=  (d^{\ell_n}_l - d^{\ell}_l)p_{\varepsilon, \xi\xi} {\color{red}-} (b^{\ell_n}_l - b_l^{\ell})p_{\varepsilon, \xi}.
       \]

		 Recall that $(p_0,q_0)\in W_0^{1,4}(0,\ell_0) \times W_0^{1,4}(\ell_0,L)$. Therefore, arguing as in Section \ref{regularityy} and Lemma \ref{lemmaregul}, we first deduce that $(F^{n}_{l},F^{n}_{r}) \in L^{4}((\ell_l,\ell_0)\times (0,T))\times L^{4}((\ell_0,\ell_r)\times (0,T))$
		and $(y_{\xi}^{n}(\ell_0,\cdot),z_{\xi}^{n}(\ell_0,\cdot))\in C^{1/8}([0,T])\times C^{1/8}([0,T])$ and also, that 
		\[
        \|y_{\xi}^{n}(\ell_0,\cdot)\|_{C^{1/8}} + \|z_{\xi}^{n}(\ell_0,\cdot)\|_{C^{1/8}} \leq C\left( \|(F^{n}_{l},F^{n}_{r})\|_{L^{4}(L^{4})\times L^{4}(L^{4})} + \|(y^n,z^n)\|_{L^2(H^2)\times L^2(H^2)}\right)
\]
       for some $C > 0$, independent of $n$.
       
		It is not difficult to check that, in this inequality, the first term in the right hand side go to $0$ when $n\rightarrow \infty$. From standard parabolic estimates applied to \eqref{difference} and \eqref{convergencecontrol}, we also have the convergence to zero of the second term in the right hand side. 
		Therefore, we deduce that $(p^{n}_{\varepsilon,\xi}(\ell_0,\cdot),q^{n}_{\varepsilon,\xi}(\ell_0,\cdot)) \rightarrow (p_{\varepsilon, \xi}(\ell_0,\cdot),q_{\varepsilon, \xi}(\ell_0,\cdot))$ in $C^{1/8}([0,T])$, which implies the continuity of $\Lambda_{\varepsilon}$. 

	\item  $\Lambda_{\varepsilon}$ {\it is compact}. Note that $\Lambda_{\varepsilon}(\ell)'(t) =  \theta(t)$ for all $\ell \in \mathcal{A}_{R}$ and all $t\in [0,T]$, where $\theta$ is the function defined in \eqref{thetadefi}. Thus, we conclude easily from \eqref{am104} that  $\Lambda_{\varepsilon}(\mathcal{A}_{R})$ is a bounded subset
	of $C^{1 + 1/8}([0,T])$, which is a compact subset of $C^{1}([0,T])$. 
       
        \item There exists $\delta > 0$, such that, whenever $(p_0,q_0) \in  W_0^{1,4}(0,\ell_0)\times  W_0^{1,4}(\ell_0,L)$ and
\[        
        \|(p_0,q_0)\|_{W_0^{1,4}\times W_0^{1,4}} + |\ell_0 -\ell_T|\leq \delta,
\]        
          then $\Lambda_{\varepsilon}(\mathcal{A}_R) \subset \mathcal{A}_R$. Indeed, it follows easily from \eqref{am104} that there exists $C > 0$ (depending on $\ell_l,\ell_r,\omega_l,\omega_r,T$ and $R$) such that
        \[
     |\mathcal{L}(t) - \ell_0| \leq CT\left(\|(p_0,q_0)\|_{ W_0^{1,4}\times  W_0^{1,4}} + |\ell_0 -\ell_T|\right)\quad \forall t\in [0,T], 
        \]
        and
        \[
        |\mathcal{L}'(t)|\leq  C\left(\|(p_0,q_0)\|_{ W_0^{1,4}\times  W_0^{1,4}} + |\ell_0 -\ell_T|\right)\quad \forall t\in [0,T].
        \]
 	Thus, we get the result by taking $\delta \leq \min\left\{\dfrac{R}{C},\dfrac{\ell_0 - \tilde{\ell}_l}{CT},\dfrac{\tilde{\ell}_r - \ell_0 }{CT}\right\}$.  
        \end{itemize}
     Consequently, for initial data $p_0, q_0$ and $\ell_0$ satisfying the above conditions, Schauder's Fixed-Point Theorem guarantees that there exists $\ell_{\varepsilon}\in \mathcal{A}_R$ such that $\Lambda_{\varepsilon}(\ell_{\varepsilon}) = \ell_{\varepsilon}$. It is easy to see that this is suffices to achieve the proof of the result.  
	\end{proof}
	
     Now, we are in conditions to prove Theorem \ref{twophase1}. Indeed, since the fixed-points $\ell_{\varepsilon}$ and controls $(h^{\varepsilon}_{l},h^{\varepsilon}_{r})$ furnished by the Theorem \ref{twophase1_approx} are uniformly bounded, respectively, in $C^{1 + 1/8}([0,T])$ and $L^2(\mathcal{O}_{l})\times L^2(\mathcal{O}_{r})$, then there exist $\ell$ and $(h_{l},h_{r})$ such that, at least for a subsequence, we have 
\begin{equation}\label{convergencevarepsilon}
	      \left\{
          \begin{array}{l}
          \ell_{\varepsilon} \rightarrow \ell~~\mbox{strongly in}~~C^1([0,T])~~\mbox{and}\\
           (h^{\varepsilon}_{l},h^{\varepsilon}_{r}) \rightarrow  (h_{l},h_{r})~~\mbox{weakly in}~~L^2(\mathcal{O}_{l})\times L^2(\mathcal{O}_{r}).
          \end{array}
          \right.
\end{equation}
		Since the coefficients $(d^{\ell_{\varepsilon}}_{l},b^{\ell_{\varepsilon}}_{l})$ and $(d^{\ell_{\varepsilon}}_{r},b^{\ell_{\varepsilon}}_{r})$ are uniformly bounded, respectively, in $L^{\infty}(Q_{0,l})\times L^{\infty}(Q_{0,l})$ and  $L^{\infty}(Q_{0,r})\times L^{\infty}(Q_{0,r})$, we can conclude from energy estimates
		and \eqref{convergencevarepsilon}  that there exists $(p,q)$ with
\begin{equation}\label{convergencevarepsilon_1}
	      \left\{
          \begin{array}{lll}
          p_{\varepsilon} \rightarrow p & \mbox{weakly in} & L^2(0,T;H^2(0,\ell_0)\cap H_0^1(0,\ell_0))\cap H^1(0,T;L^2(0,\ell_0)) ,\\
            q_{\varepsilon} \rightharpoonup q & \mbox{weakly in} & L^2(0,T;H^2(0,\ell_0)\cap H_0^1(\ell_0,L))\cap H^1(0,T;L^2(\ell_0,L)),
          \end{array}
          \right.
\end{equation}
	where the $(p_{\varepsilon},q_{\varepsilon})$ are associated to the $(h^{\varepsilon}_{l},h^{\varepsilon}_{r})$.
		Then, $(p,q)$ is the solution to \eqref{am7}, associated to $(h_{l},h_{r})$. Moreover, 
		from \eqref{approximatedcontrol2}, it is clear that $\ell(T) = \ell_T$ and $(p(\cdot,T),q(\cdot,T)) = (0,0)$ on $(0,T)$.
		
	Furthermore, as a consequence of \eqref{convergencevarepsilon_1} and the embeddings
\begin{equation*}
       \begin{alignedat}{2}
       H^2(0,\ell_0) \stackrel{c}{\hookrightarrow} C^1([0,\ell_0]) \hookrightarrow L^2(0,\ell_0)\quad \hbox{and}\quad 
       H^2(\ell_0,L) \stackrel{c}{\hookrightarrow} C^1([\ell_0,L]) \hookrightarrow L^2(\ell_0,L),
       \end{alignedat}
\end{equation*}
	we find that, for any given $t \in [0,T]$, the following holds:
\begin{equation*}
\begin{alignedat}{2}
\ell(t) = \lim_{\varepsilon\rightarrow 0}\ell_{\varepsilon}(t) & = \lim_{\varepsilon\rightarrow 0}   \left(\ell_0 - \displaystyle\int_{0}^{t}\left[d_lp_{\varepsilon,\xi}(\ell_0,\tau) - d_rq_{\varepsilon,\xi}(\ell_0,\tau)\right]\,d\tau\right)\\
                     & = \ell_0 - \displaystyle\int_{0}^{t}\left[d_lp_{\xi}(\ell_0,\tau) - d_rq_{\xi}(\ell_0,\tau)\right]\,d\tau.
\end{alignedat}
\end{equation*}	
	This implies the Stefan condition \eqref{am4}$_7$ is satisfied by $(\ell,p,q)$ and ends the proof of Theorem \ref{twophase1}.


 \section{Additional comments} \label{sec:comments}

   
\subsection{Lack of controllability with only one control}\label{sec:lack}
 
 	In the next result it is proved that, if $h_{l}$ or $h_{r}$ vanishes and the interface does not collapse to the boundary, then null controllability cannot hold. 
\begin{thm}\label{twophase0}
	Assume that  $u_0\in H_0^1(0,\ell_0)$ with $u_0 \geq 0$, $v_0\in H_0^1(\ell_0,L)$ with $v_0 \leq 0$ and 
	$v_0\not\equiv0$. Then, if $(h_l,h_r)\in L^2(\mathcal{O}_l)\times L^2(\mathcal{O}_r)$, $h_r\equiv0$ and 
	the associated strong solution to \eqref{am} satisfies $\ell(t) < L$ for all $t \in [0,T]$, we necessarily have
	\begin{equation*}\label{am20}
	v(\cdot,T) \not \equiv 0~~\mbox{in}~~(\ell(T),L).  
\end{equation*}
\end{thm}
\begin{proof}
	 Let us assume, by contradiction, that \eqref{am} is null-controllable with $h_r \equiv 0$, i.e. 
	$u(\cdot,T) \equiv 0$ in $(0,\ell(T))$ and $v(\cdot,T) \equiv 0$ in $(\ell(T), L)$.
	
	Then, considering the diffeomorphism $\Phi$ and the function $q = v \circ \Phi^{-1}$, defined in Section~\ref{reformulation}, we get easily that $q$ is the solution to
\begin{equation*}\label{lackcontroll}
        \left\{
        \begin{array}{lll}
        q_t - d_r^{\ell}q_{\xi\xi} + b_r^{\ell}q_{\xi} = 0 & \mbox{in} & Q_{0,r},\\
        q(\ell_0,\cdot) = q(L,\cdot) = 0 & \mbox{on} & (0,T),\\
        q(\cdot,0) = q_0  & \mbox{in} & (\ell_0, L),
        \end{array}
        \right.
\end{equation*}	 
	where $q_0 := v_0 \circ \left[G(\cdot,\ell_0)\right]^{-1} \in H_0^1(\ell_0,L)$ and, obviously, $q_0\leq 0$ and $q_0\not\equiv 0$. We also have
\begin{equation}\label{lackcontroll_1}
	    q(\cdot,T) \equiv 0 ~~\mbox{in}~~(\ell_0,L).
\end{equation}
      We can apply the weak maximum principle to $q$ in $Q_{0,r}$. Obviously, this gives $q \leq 0 $ in $Q_{0,r}$.
   
  	Thus, if \eqref{lackcontroll_1} holds we get, from the strong maximum principle and from the fact $q \leq 0 $ in $Q_{0,r}$, that $q \equiv 0$ in $\overline{Q}_{0,r}$, what contradicts $q_0 \not\equiv 0$.
\end{proof}


  \begin{remark}\label{rmq:2}
   	{\rm Note that the previous argument also shows that the null controllability for~\eqref{am} cannot be achieved keeping the signs of the initial conditions in each phase region.
	In other words, in order to drive the solution to zero at time $T$, the liquid and solid state must penetrate each other before $T$.
	\Fin}
	\end{remark} 

\subsection{Boundary controllability and other extensions}

	We can prove local boundary controllability results similar to Theorem~\ref{twophase}.
	Thus, let us introduce the system
        \begin{equation}\label{am:boundary}
	\left\{
		\begin{array}{lll}
			u_t - d_lu_{xx} = 0 	&\mbox{in}& 	Q_l,		\\
			\noalign{\smallskip}\displaystyle	
      			v_t - d_rv_{xx} = 0	&\mbox{in}& 	Q_r,		\\
			\noalign{\smallskip}\displaystyle	
      			u(0,t)  = k_l(t),~ v(L,t)  = k_r(t)  				&\mbox{on}& 	(0,T),	\\			
		  \noalign{\smallskip}\displaystyle	
      			u(\cdot,0) = u_0  				&\mbox{in}& 	(0,\ell_0),	\\ 
			\noalign{\smallskip}\displaystyle	
      			v(\cdot,0) = v_0 				&\mbox{in}& 	(\ell_0,L), 	\\
			\noalign{\smallskip}\displaystyle
				u(\ell(t),t) =v(\ell(t),t) =0   &\mbox{on}& 	(0,T),	\\
			\noalign{\smallskip}\displaystyle	  		
			- \ell'(t) = d_lu_x(\ell(t),t) - d_rv_{x}(\ell(t),t)	&\mbox{on}&	(0,T).
		\end{array}
	\right.
        \end{equation}
where $(k_l,k_r)$ stands for the boundary control pair.
	
	Then, using a domain extension technique and Theorem~\ref{twophase}, it is easy to prove that, if $u_0$ and $v_0$ are sufficiently small, and $\ell_0$ is sufficiently close to $\ell_T$, there exist controls $(k_l,k_r)$ and associated solutions to~\eqref{am:boundary} that satisfy $\ell(T) = \ell_T$, $u(\cdot,T) = 0$ in $(0,\ell_T)$ and $v(\cdot,T) = 0$ in $(\ell_T,L)$.
	
	 Let us finally mention that the arguments and results in this paper can also be used to solve other variants of the two-phase Stefan controllability problem.
	Thus, we can prove results similar to Theorem~\ref{twophase} when the controls are Neumann data, we can assume that the equations contain lower order terms or even appropriate nonlinearities, etc.

%
%


\end{document}